\documentclass[11pt,a4paper, reqno]{amsart}
\usepackage{mathrsfs}
\usepackage[utf8]{inputenc}

\usepackage{amsmath,amsfonts,verbatim}
\usepackage{latexsym}
\usepackage{amssymb,leftidx}
\usepackage{extarrows}
\usepackage{overpic}
\usepackage{color}
\usepackage{epsfig}
\usepackage{subfigure}
\usepackage{tikz}
\usepackage{bbm}
\usepackage{tikz, caption}
\usepackage[font=small,labelfont=bf]{caption}
\usepackage{enumerate}
\usepackage[backref=page]{hyperref}
\hypersetup{urlcolor=blue, citecolor=red}
\usepackage{hyperref}

\usepackage{amssymb}
\usepackage{mathrsfs}
\usepackage{amscd}
\usepackage{cite}

\usepackage{graphicx}
\usepackage{float}
\newcommand{\Ss}{{\mathbb{S}}^1_{\sigma}}
\newcommand{\R}{\mathbb{R}}
\newcommand{\Z}{\mathbb{Z}}

\newcommand{\C}{\mathbb{C}}

\newcommand{\D}{\mathcal{D}}
\newcommand{\Pk}{\mathcal{P}_k}
\newcommand{\bra}[1]{\langle #1 \rangle}
\newcommand{\hank}{\mathcal{H}_\nu}

\numberwithin{equation}{section}
\newtheorem{proposition}{Proposition}[section]

\newtheorem{lemma}{Lemma}[section]
\newtheorem{theorem}{Theorem}[section]
\newtheorem{corollary}{Corollary}[section]
\newtheorem{remark}{Remark}[section]

\newenvironment{equ*}{\begin{equation*}}{\end{equation*}}

\begin{document}
\title[Dirac equation in the cosmic string spacetime]{Dispersive and Strichartz estimates for Dirac equation in a cosmic string spacetime}

\author{Piero D'Ancona}
\address{Piero D'Ancona:
Aldo Moro 5, 00185 Rome, Italy
Department of Mathematics "Guido Castelnuovo", Sapienza University of Rome, Piazzale Aldo Moro 5, 00185 Rome, Italy}
\email{dancona@mat.uniroma1.it}

\author{Zhiqing Yin}
\address{Zhiqing Yin:
Department of Mathematics, Beijing Institute of Technology, Beijing 100081; Department of Mathematics "Guido Castelnuovo", Sapienza University of Rome, Piazzale Aldo Moro 5, 00185 Rome, Italy}
\email{zhiqingyin@bit.edu.cn}

\author{Junyong Zhang}
\address{Junyong Zhang:
Department of Mathematics, Beijing Institute of Technology, Beijing 100081}
\email{zhang\_junyong@bit.edu.cn}
\maketitle

\begin{abstract}
  In this work we study the Dirac equation on the cosmic string background, which models a one--dimensional topological defect in the spacetime. We first define the Dirac operator in this setting, classifying all of its selfadjoint extensions, and we give an explicit kernel for the propagator. Secondly, we prove dispersive estimates for the flow, with and without weights. Finally, we prove Strichartz estimates for the flow in a sharp restricted set of indices, which are different from
the classical Euclidean ones.
\end{abstract}

\begin{center}
 \begin{minipage}{120mm}
 \end{minipage}
 \end{center}

\tableofcontents
\section{Introduction}


A \emph{cosmic string} is a one--dimensional defect line in spacetime that might have formed during the early phases of the universe (see \cite{Paker,VS}). The existence of such objects is predicted by some grand unified theories. In the simplest axisymmetric model, the string lies along the $z$--axis and the defect is described by the metric
\begin{equation*}
  ds^{2}=dt^{2}-dz^{2}-dr^{2}-r^{2}d \theta^{2},
  \qquad
  0\le \theta<2\pi(1-4G \mu)
\end{equation*}
where $G$ is Newton's gravitational constant and $\mu$ is the `unit mass' of the string.
We see that the metric is locally flat but the angular coordinate is restricted to a range smaller than $[0,2\pi]$, giving rise to a flat cone metric in the $(r,\theta)$ plane
 \cite{BFM, CT1, CT2, Ford, YZ, Zhang1}.

In this work we study the Dirac equation on a cosmic string background. We are mainly interested in the effect of the topological defect on the dispersive properties of the Dirac flow. Since the singularity is present only in the transverse dimensions, one can ignore the $z$ variable and consider the reduced 2D metric (see \cite{GA1984})
\begin{equation}\label{metric}
  dl^2=dr^2+r^2d\theta^2,\quad -\pi \sigma\leq \theta<\pi \sigma,
\end{equation}
where we have set $0\leq1-\sigma=4G \mu<1$ and translated the angular variable.



Dispersive equations perturbed by electromagnetic fields or on a curved background have inspired a great deal of research in recent years. We mention in particular the Aharonov--Bohm (AB) field \cite{AB59}, where a charged particle can interact with the field also in those regions with vanishing field. This two dimensional model has a strong similarity with the axisymmetric cosmic string.
Indeed, in both models the singularity is concentrated at the origin, and at least formally the field (for the AB model) or the curvature (for the cosmic string model) is zero away from the origin; however, a physical effect can be measured. In the first case, this is precisely the content of the AB experiment \cite{AB59}, while
e.g.~in \cite{FRV} the effects of a cosmic string spacetime away from the singularity were studied. The dispersive properties of the Dirac equation perturbed by the Aharonov--Bohm field were studied by the authors of the present paper in \cite{CDYZh}.

In physics, the Dirac oscillator has been studied in various contexts including the Aharonov--Bohm field \cite{FB}, chiral phase transitions under a constant magnetic field \cite{BML}, and, more recently, the cosmic string background.
For the Dirac oscillator in a cosmic string spacetime we refer to \cite{Bak2012,Bak2013,BM2018,CFM}, for spinning cosmic strings to \cite{GJ1989,HHM}, and for magnetic cosmic strings in \cite{AS2014, ASP}. These works motivated us to investigate the decay of Dirac flows in a cosmic string spacetime.


\subsection{Dirac operator on the cosmic string}
For a complete derivation of the Dirac equation in curved space-time, we refer to Section 5.6 in \cite{Paker}.
The concrete form we adopt here is presented
also in \cite[Section 2.1]{BCSZh}.

We use the choice of matrices
\begin{equation}\label{m-Pauli}
  \gamma^0=\begin{pmatrix}
        1&0\\
        0&-1
        \end{pmatrix},
        \quad
  \gamma^1=\begin{pmatrix}
         0&-i\\
         -i&0
         \end{pmatrix},
      \quad
  \gamma^2=\begin{pmatrix}
         0&-1\\
         1&0
         \end{pmatrix}
\end{equation}
satisfying the usual anticommutation relations (see \cite{Thaller})
$\{\gamma^{a},\gamma^{b}\}=2 \eta^{ab}$,
$\eta=\mathop{\rm diag}[+1,-1,-1]$.
Recall that the metric of $X$ in global radial coordinates
$(x_{1},x_{2})=(r,\theta)$ is given by
$dl^2=dr^{2}+r^{2}d \theta$;
we denote derivatives w.r.t. these coordinates by
$(\partial_{1},\partial_{2})=(\partial_{r},\partial_{\theta})$.
We introduce the \emph{zweibein} formalism:
$e_{1},e_{2}$ are the orthogonal vector fields
\begin{equation*}
  e_{1}=\partial_{r},\qquad
  e_{2}=\frac 1r \partial_{\theta}
\end{equation*}
and the corresponding dual forms are
\begin{equation*}
  e^{1}=dr,\qquad
  e^{2}=r d \theta.
\end{equation*}
Writing $e^{a}=e^{a}_{j}dx^{j}$, the matrix $e^{a}_{j}$
takes the form
\begin{equation*}
  \begin{pmatrix}
    1 & 0 \\
    0 & r
  \end{pmatrix}.
\end{equation*}
The Dirac operator on $X$ is defined as
$\mathcal{D}= e^j_a \gamma^a D_j$, where
$D_{j}=\partial_{j}+B_{j}$ are
the covariant derivatives for spinors
(see \cite[Section 5.6, page 226 (5.270)]{Paker}); more explicitly,
\begin{equation*}
  B_{j}=\frac 18 \omega^{ab}_{j}[\gamma^{a},\gamma^{b}],
  \qquad
  \omega^{ab}_{j}=e^{a}_{i}\Gamma^{i}_{jk}e^{kb}+
    e^{a}_{i}\partial_{j}e^{ib}.
\end{equation*}
$B_{j}$ contains an algebraic part $[\gamma^{a},\gamma^{b}]$,
corresponding to the generators of the underlying Lie algebra
for Dirac bispinors, and the spin connection $\omega^{ab}_{j}$,
which can be computed via Cartan's first structural equation
\begin{equation*}
  de^{a}+\omega^{a b}\wedge e^{b}=0,
  \qquad
  \omega^{ab}=\omega^{ab}_{j} dx^{j}.
\end{equation*}
Since $de^1=0$, $de^2=\frac1r e^1\wedge e^2=-\frac1r e^2\wedge e^1$,
we get $\omega^{12}=-\frac 1r e^2$ and $\omega^{21}=\frac1r e^2$,
and $\omega^{12}_2=-1$ and $\omega^{21}_2=1$. It follows
that $B_{1}=0$ and
\begin{align*}
B_2&=\frac18\omega^{ab}_2[\gamma^a,\gamma^b]\\
&=\frac18\omega^{12}_2[\gamma^1,\gamma^2]+\frac18\omega^{21}_2[\gamma^2,\gamma^1]\\
&=\frac12\omega^{12}_2\gamma^1 \gamma^2
\qquad
\text{(since $[\gamma^1,\gamma^2]=2\gamma^1 \gamma^2$)}\\
&=-\frac12\gamma^1 \gamma^2.
\end{align*}
We obtain
\begin{equation*}
  e^j_a\gamma^a D_j=\gamma^1D_1+\frac1r\gamma^2 D_2
  =\gamma^1\partial_r+\frac1r(\gamma^2\partial_\theta+\frac12\gamma^1)
\end{equation*}
and in conclusion the Dirac operator takes the form
\begin{equation}\label{eq:diracop}
  \mathcal{D}=
  \gamma^1\left(\partial_r+\frac1{2r}\right)+
    \frac1r \gamma^2\partial_\theta.
\end{equation}

\noindent The time dependent Dirac equation on $X$ is then
\begin{equation}\label{eq:dirac}
  [\gamma^{0}\partial_{t}+\mathcal{D}]u
  \equiv
  \left[\gamma^0\partial_t+\gamma^1\left(\partial_r+\frac1{2r}\right)
  +\frac1r \gamma^2\partial_\theta\right]u=0.
\end{equation}
Multiplying the equation by $i\gamma^{0}$ we obtain the
equivalent Hamiltonian form
\begin{equation*}
  i\partial_t u=\mathcal{D}_\sigma u,
\end{equation*}
where the operator $\mathcal{D}_\sigma$ now is given by
\begin{equation}\label{eq:convention}
  \mathcal{D}_\sigma=
  \gamma^{2}\left(\partial_r+\frac1{2r}\right)
  -\frac1 r \gamma^{1} \partial_\theta.
\end{equation}
We shall denote by $\mathcal{D}_{\mathbb{S}_\sigma^1}$
the Dirac operator on $\mathbb{S}_\sigma^1:=\R/ 2\pi \sigma\Z$ is a circle of radius $0<\sigma\leq 1$:
 \begin{align}
 \mathcal{D}_{\mathbb{S}_\sigma^1}=-\gamma^1\partial_\theta=\begin{pmatrix}
 0&i\partial_\theta\\
 i\partial_\theta&0\end{pmatrix},\quad D_{\mathbb{S}_\sigma^1}=i\partial_\theta.
 \end{align}
Then the operator $\mathcal{D}_{\sigma}$ can be written as
 \begin{align}
 \mathcal{D}_\sigma=\begin{pmatrix}0&-(\partial_r+\frac1{2r})+\frac 1r D_{\mathbb{S}_\sigma^1}\\(\partial_r+\frac1{2r})+\frac 1r D_{\mathbb{S}_\sigma^1}&0
 \end{pmatrix}.
 \end{align}
The operator $\mathcal{D}_{\mathbb{S}_\sigma^1}$ satisfies
 \begin{align}\label{eq:a-eigen}
 \mathcal{D}_{\mathbb{S}_\sigma^1}\varphi_k(\theta)=\lambda_k \varphi_k(\theta),\quad  \varphi_k(\sigma\pi)=\varphi_k(-\sigma\pi),
 \end{align}
where $\lambda_k=\frac k\sigma$ and
 \begin{align}\label{varphi}
 \varphi_k(\theta)=\Big(\begin{smallmatrix}e^{-i\frac k\sigma \theta}\\e^{-i\frac k\sigma \theta}\end{smallmatrix}\Big)
  \end{align} are eigenvalues and eigenfunctions of the
(obviously diagonizable)
operator $\mathcal{D}_{\mathbb{S}_\sigma^1}$.
This leads to the decomposition of $[L^{2}(\R^2)]^2$ as
\begin{equation}\label{L2Domc}
  [L^{2}(\R^2)]^2=
  \bigoplus_{k\in \mathbb{Z}}
  L^{2}(rdr)^2\otimes h_{k}(\mathbb{S}_\sigma^{1}),
\end{equation}
where $h_{k}(\mathbb{S}_\sigma^{1})$ is the one dimensional space
\begin{equation*}
  h_{k}=
  h_{k}(\mathbb{S}_\sigma^{1})=
  \left[
    \begin{pmatrix}
      e^{-i\frac k\sigma \theta} \\
      e^{-i\frac k\sigma\theta}
    \end{pmatrix}
  \right]=
  \left\{
    c\begin{pmatrix}
      e^{-i\frac k\sigma \theta} \\
      e^{-i\frac k\sigma\theta}
    \end{pmatrix}:
    c\in \mathbb{C}
  \right\}.
\end{equation*}
Then any
$f=(\begin{smallmatrix} \phi \\ \psi \end{smallmatrix})
  \in \big[L^{2}(X)\big]^{2}$
can be expanded in the form
\begin{equation*}
  \phi=\sum_{k\in \mathbb{Z}}\phi _{k}(r)e^{-i\frac k\sigma \theta},
  \qquad
  \psi=\sum_{k\in \mathbb{Z}}\psi _{k}(r)e^{-i\frac k\sigma\theta}.
\end{equation*}
The corresponding decomposition of $\mathcal{D}_{\sigma}$
is given by
\begin{align}
\D_{\sigma}=\bigoplus_{k\in\Z} d_k\otimes I_2, \quad
 I_2:=\mathrm{Id}
 \begin{pmatrix}
   1 & 0 \\
 0 & 1
 \end{pmatrix},
\end{align}
where
 \begin{align}\label{eq:dk}
 d_{k}=\begin{pmatrix}0&-(\partial_r+\frac1{2r})+\frac{k}{\sigma r}\\
 (\partial_r+\frac1{2r})+\frac{k}{\sigma r}&0
 \end{pmatrix},
 \end{align}
and we are reduced to study
the dispersive properties of the equations
\begin{equation*}
 i\partial_t u_k-d_k u_k=0,
 \qquad k\in \mathbb{Z}.
\end{equation*}

\subsection{Results}\label{sec:results}

In the sequel we shall first introduce the following notations:
\begin{itemize}
  \item
  We denote derivatives w.r.to $r$ by
  $f'(r)=\frac{d}{dr}f(r)=\partial_{r}f(r)$.
  \item
  $P_{0}$ and $P_{\perp}$ are the orthogonal projections on
  $L^2(\R^2)^2$ defined by
  \begin{equation*}
    P_0: L^2(\R^2)^2\rightarrow
    L^2(rdr)^2\otimes h_0(\mathbb{S}^1_{\sigma}),\quad
    P_{\bot}=I-P_0.
  \end{equation*}
  \item
  We further split $P_{\bot}=P_> + P_<$, where
  $P_>$ and $P_<$ are the orthogonal projections
  with respect to \eqref{L2Domc} relative to $k>0$
  and $k<0$ respectively:
  \begin{align}
    P_>: L^2(\R^2)^2\rightarrow
    \bigoplus_{0<k\in\Z}L^2(rdr)^2
    \otimes h_k(\mathbb{S}^1_{\sigma}),\\
    P_<: L^2(\R^2)^2\rightarrow
    \bigoplus_{0>k\in\Z}L^2(rdr)^2\otimes h_k(\mathbb{S}^1_{\sigma}).
    \end{align}
  \item
  We fix a standard dyadic decomposition of unity
  $\sum_{k\in\Z}\varphi(2^{-j}\lambda)=1$ for a suitable
  $\varphi\in C^\infty_c(\R\setminus\{0\})$
  with $0\le \phi\le1$, supported in $[\frac12,1]$
  and write
  \begin{equation}\label{dy-dec}
    \varphi_j(\lambda):=\varphi(2^{-j}\lambda),\quad j\in\Z,
    \qquad \phi_0(\lambda):=\sum_{j\leq0}\varphi(2^{-j}\lambda)
  \end{equation}
  Moreover we fix $\tilde{\varphi}\in C^\infty_c([\frac12,4])$
  with $\varphi\tilde{\varphi}=\varphi$.
\end{itemize}

In Section \ref{sec:sa} we shall classify all selfadjoint
extensions of the operator $\mathcal{D}_{\sigma}$; these are
indicized by a real parameter $\gamma\in[0,2\pi)$ and we denote
them by $\mathcal{D}_{\sigma,\gamma}$.
In particular, we can define spectral localizations
$\phi(2^{-j}|\mathcal{D}_{\sigma,\gamma}|)$ via the spectral
theorem.
Note that dispersion holds only for distinguished
selfadjoint extensions, corresponding to the special choices
of the parameter $\gamma$ that satisfy
$\sin \gamma\cos \gamma=0$.

Our first main result is a $L^{1}-L^{\infty}$ dispersive
estimate. In order to avoid unnecessary technicalities, we
localize the estimates in frequency, using the spectral
projections related to $\mathcal{D} _{\sigma,\gamma}$.
For the nonsingular part of the flow we obtain estimates with
the same decay rate and regularity loss as for the
euclidean 2D Dirac equation. However, for the singular
component (corresponding to the projection $P_{0}$)
we are only able to prove a weighted dispersive estimate.
To this end we introduce the following matrix weights,
depending on the special choice of $\gamma$:
\begin{equation}\label{weig-Wj1}
W_j(|x|)=\begin{cases}
\left(\begin{matrix} (1+2^j|x|^{-\frac12})^{-1}& 0 \\0 &  1\end{matrix}\quad \right),\quad \sin\gamma=0;\\
\left(\begin{matrix} 1& 0 \\0 &  (1+2^j|x|^{-\frac12})^{-1}\end{matrix}\quad \right),\quad \cos\gamma=0.
\end{cases}
\end{equation}

\begin{theorem}[$L^{q'}-L^q$-time decay estimates]\label{th:disp}
  Let $\sigma\in(0,1]$,
  $\gamma\in[0,2\pi)$ such that $\sin\gamma\cos\gamma=0$
  and let $\D_{\sigma,\gamma}$ be the corresponding
  selfadjoint extension of $\D_{\sigma}$.
  Then for any $f\in[L^1(X)]^2\cap [L^2(X)]^2$ and any $j\in\Z$, there exists a constant $C$ such that
 \begin{equation}\label{est:Pk}
 \begin{aligned}
 \Big\|e^{it\D_{\sigma,\gamma}}&\varphi(2^{-j}|\D_{\sigma,\gamma}|) P_{\bot}f(x)\Big\|_{[L^\infty(X)]^2}\\
 &\leq C2^{2j}(1+2^{j}|t|)^{-\frac12}\|\tilde{\varphi}(2^{-j}| \D_{\sigma,\gamma}|)P_{\bot}f\|_{[L^1(X)]^2},
 \end{aligned}
 \end{equation}
and
 \begin{equation}\label{est:P0}
 \begin{aligned}
 \Big\|W_j(|x|)&e^{it\D_{\sigma,\gamma}}\varphi(2^{-j}|\D_{\sigma,\gamma}|) P_{0}f(x)\Big\|_{[L^\infty(X)]^2}\\
 &\leq C2^{2j}(1+2^{j}|t|)^{-\frac12} \|\tilde{\varphi}(2^{-j}| \D_{\sigma,\gamma}|)P_{0}f\|_{[L^1(X)]^2},
 \end{aligned}
 \end{equation}
Furthermore, if $2\leq q<4$, we have the
(weightless) decay estimates
 \begin{equation}\label{est:P0q}
 \begin{aligned}
 \Big\|e^{it\D_{\sigma,\gamma}}&\varphi(2^{-j}|\D_{\sigma,\gamma}|) P_0f(x)\Big\|_{[L^q(X)]^2}\\
 &\leq C 2^{2j(1-\frac2q)}(1+2^j|t|)^{-\frac12(1-\frac2q)}\| \tilde{\varphi}(2^{-j}|\D_{\sigma,\gamma}|)P_0f\|_{[L^{q'}(X)]^2}.
 \end{aligned}
 \end{equation}
\end{theorem}

Since the regular component $P_{\perp}f$ satisfies the same
decay estimates as in the euclidean setting, by standard
techniques we deduce Strichartz estimates for the full
range of indices.
Following \cite[Section 2]{BFM}, we use Sobolev spaces
defined via functional calculus in terms of the
Dirac square operator $H=\mathcal{D}_{\sigma}^{2}$ on $X$
(see \eqref{eq:square})
\begin{equation*}
  \mathbf{H}^{s}(X)=\{f\in[L^{2}(X)]^{2}:
    H^{s/2}f\in[L^{2}(X)]^{2}\}
\end{equation*}
with norm $\|u\|_{\mathbf{H}^{s}}=\|(1+H^{s/2})u\|_{L^{2}}$,
and its homogeneous version $\dot {\mathbf{H}}^{s}(X)$,
which can be defined e.g.~as
the completion of test functions in the homogeneous norm
$\|H^{s/2}f\|_{L^{2}}$
(see also \cite[Section 2]{BFM}).

\begin{theorem}
  \label{th-stri1}
  Let $\mathcal{D}_{\sigma,\gamma}$ be the
  self--adjoint extension of $\mathcal{D}_\sigma$ selected in
  Theorem \ref{th:disp}.
  Then for any $f\in \dot{\mathbf{H}}^{s}(X)$ and $p>2$,
  the following Strichartz estimates hold
  \begin{equation}\label{stri-D1}
    \|e^{it\mathcal{D}_{\sigma,\gamma}}P_{\bot}f\|
    _{[L^p_t(\R; L^q_x(X))]^2}
    \leq C \| f\|_{\dot{\mathbf{H}}^{s}(X)},
    \qquad s=1-\frac1p-\frac2q
  \end{equation}
  provided $(p,q)$ satisfy
  \begin{equation}\label{pqrange1}
    (p,q)\in (2,\infty]^2
    \quad\text{and}\quad
    \frac2p+\frac1q\leq\frac12,
    \qquad \text{or}\quad
    (p,q)=(\infty, 2).
  \end{equation}
\end{theorem}
The hard case of course are Strichartz estimates for the singular component $P_0f$ in view of the singularity at the origin.
Note that the component $k=0$ is actually a radial component,
hence the range \eqref{pqrange0}
of admissible $(p,q)$ for the estimates is
larger, as it is usual for the radial case.

\begin{theorem}
  \label{th-stri3}
  Let $\mathcal{D}_{\sigma,\gamma}$ be the
  self--adjoint extension of $\mathcal{D}_\sigma$ selected in
  Theorem \ref{th:disp}. Assume the pairs $(p,q)$ satisfy
  \begin{equation}\label{pqrange0}
  (p,q)\in (2,\infty]^2 \quad{\rm and}\quad \frac1p+\frac1q<\frac12, \qquad \text{or}\quad (p,q)=(\infty, 2).
  \end{equation}
  Assume in addition that $q<4$.
  Then for any $f\in \dot{\mathbf{H}}^{s}(X)$ we have the following
  Strichartz estimates:
  \begin{equation}\label{stri-D3}
    \|e^{it\mathcal{D}_{\sigma,\gamma}}P_{0}f\|
    _{[L^p_t(\R; L^q_x(X))]^2}
    \leq C \|P_{0} f\|_{\dot{\mathbf{H}}^{s}(X)},
    \qquad
    s=1-\frac 1p-\frac 2q.
  \end{equation}
\end{theorem}

\begin{remark}\label{rem:q=4}   The restriction $q<4$ is necessary, see Remark \ref{rem:q<4} below. However, one can recover $q=4$ by proving
  \begin{equation}\label{stri-D3'}
    \|e^{it\mathcal{D}_{\sigma,\gamma}}P_{0}f\|
    _{[L^4_t(\R; L^{4,\infty}_x(X))]^2}
    \leq C \|P_{0} f\|_{\dot{\mathbf{H}}^{\frac14}(X)},
  \end{equation}
where we replace the Lebesgue space $L^{4}$ by the Lorentz space $L^{4,\infty}$. By interpolation, one can prove a weak estimate in Lorentz norms along the board line
$1/p+1/q=1/2$. See \eqref {eq:Z} below.
\end{remark}

Combining the Theorems \ref{th-stri1} and \ref{th-stri3}, we obtain the full Strichartz estimates:

\begin{corollary}[Strichartz estimates]
  Under the assumptions of Theorem \ref{th-stri1}, with the
  additional restriction $q<4$,
  the following Strichartz estimates hold:
  \begin{equation}\label{stri}
  \|e^{it\mathcal{D}_{\sigma,\gamma}}f\|_{[L^p_t(\R; L^q_x(X))]^2}\leq
  C
  \| f\|_{\dot{\mathbf{H}}^{s}(X)}.
  \end{equation}
  The restriction $q<4$ is necessary
  in the sense that \eqref{stri} may fail if  $q\geq 4$.
\end{corollary}

\begin{remark}\label{rem:q<4}
The restriction $q<4$ is sharp.
Actually, it is possible to construct an explicit counterexample,
corresponding to the choice of
data $f=(\mathcal{H}_{-1/2}\chi,0)^T$,
where $\chi\in C^\infty_0([1,2])$ takes values in $[0,1]$
and $\mathcal{H}_{-1/2}$ is the Hankel transform
(see Section \ref{sec:hank} for definitions).
Then, similarly as in \cite[Proposition 6.6]{CDYZh} with $\alpha=-1/2$, we can prove
 \begin{align*}
 \Big\|\int^\infty_0e^{it\rho}J_{-\frac12}(r\rho)\chi(\rho)\rho d\rho\Big\|_{L^p(\R;L^q(X))}=\infty,\quad q\geq4,
 \end{align*}
due to the singularity of the Bessel function $J_{-1/2}(r)$ of order $-1/2$  at $r=0$.
\end{remark}

If we include an appropriate weight in the estimates,
we can recover the full range of admissible pairs
also the singular component, like in Theorem \ref{th:disp}.
Thus we introduce the weights
\begin{equation}\label{weig-Wj}
  W(|x|)=\begin{cases}
    \begin{pmatrix}
      (1+|x|^{-\frac12-\epsilon})^{-1}& 0 \\
      0 &  1
    \end{pmatrix}
    & \sin\gamma=0,\\
    \begin{pmatrix}
      1& 0 \\
      0 &  (1+|x|^{-\frac12-\epsilon})^{-1}
    \end{pmatrix}
    & \cos\gamma=0,
\end{cases}
\end{equation}
where $0<\epsilon\ll 1$ is a fixed small parameter.
We obtain the following result:

\begin{theorem}[Weighted Strichartz estimates for the singular component]\label{th-stri2}
  Let $\mathcal{D}_{\sigma,\gamma}$ be the
  self--adjoint extension of $\mathcal{D}_\sigma$ selected in
  Theorem \ref{th:disp}. Assume $(p,q)$ satisfy \eqref{pqrange0}
  and $4\le q<\infty$. Take the weight $W(|x|)$ defined by \eqref{weig-Wj}.
  Then for any $f\in [{\mathbf{H}}^{s}(X)]^2$
  and any $\theta>1-\frac{4}{q}$ the following weighted
  Strichartz estimates hold:
  \begin{equation}\label{stri-D2b}
    \|W(|x|)^{\theta}e^{it\mathcal{D}_{\sigma,\gamma}}P_{0}f\|
    _{[L^p_t(\R; L^q_x(X))]^2}
    \leq C_{\epsilon,\theta} \|P_{0} f\|_{{\mathbf{H}}^{s}(X)},
    \qquad s=1-\frac1p-\frac2q.
  \end{equation}
\end{theorem}

\begin{remark}The conditions on $(p,q)$ are represented in the figure.
Estimates \eqref{stri-D1} hold for $(p,q)$ in the triangle $AOC$, the weighted Strichartz estimates \eqref{stri-D2b} hold for  $(p,q)$ in the triangle $AOB$.
On the other hand,
estimates \eqref{stri-D3} hold for $(p,q)$ in the triangle $ADF$
and the classical estimates \eqref{stri} for $(p,q)$ in the triangle $ADE$. Compared with the classical Euclidean ones, there is a red line $1/q=1/4$. Below this line, there is no weightless Strichartz estimates.
 
 \begin{center}
  \begin{tikzpicture}[scale=1]
 \draw[->] (4,0) -- (8,0) node[anchor=north] {$\frac{1}{p}$};
 \draw[->] (4,0) -- (4,4)  node[anchor=east] {$\frac{1}{q}$};

 \draw  (4.1, -0.1) node[anchor=east] {$O$};
 \draw  (7, 0) node[anchor=north] {$\frac12$};
 \draw  (4, 3) node[anchor=east] {$\frac12$};
 \draw  (5.5, 0) node[anchor=north] {$\frac14$};

 \draw[thick] (4,3) -- (7,0);  
 \draw[red,thick] (4,3) -- (5.5,0);
 \draw[red, dashed,thick] (4,3) -- (7,0); 

 \draw (3.9,3.15) node[anchor=west] {$A$};
 \draw (6.9,0.2) node[anchor=west] {$B$};
 \draw (5.5,-0.2) node[anchor=west] {$C$};
 \draw (6,2.6) node[anchor=west] {$\frac{1}{p}+\frac{1}{q}=\frac12$};

 \draw (6,1.6) node[anchor=west] {$\frac{2}{p}+\frac{1}{q}=\frac12$};

 \draw (7,0) circle (0.06);

 \filldraw[fill=gray!30](4,3)--(5.51,0)--(7,0); 
 \filldraw[fill=gray!50](4,3)--(5.49,0)--(4,0); 


 \draw[<-] (5,2.1) -- (6,2.6) node[anchor=south]{$~$};
 \draw[<-] (4.86,1.3) -- (6,1.6) node[anchor=south]{$~$};

 \draw[red, dashed, thick] (4,1.5) -- (5.5,1.5); \draw (4,1.5) circle (0.06); \draw (4.75,1.5) circle (0.06); \draw (5.5,1.5) circle (0.06);
 \draw  (3.7,1.85 ) node[anchor=north] {$\frac1{4}$};
 \draw (3.8, 1.2) node[anchor=west] {D};
 \draw (4.75, 1.25) node[anchor=west] {E};
 \draw (5.45, 1.2) node[anchor=west] {F};

 \path (6,-1.5) node(caption){Figure 1: Diagrammatic picture of the admissible range of $(p,q)$.};  
 \end{tikzpicture}
 \end{center}

\end{remark}

\subsection{The relation to the previous results}

The techniques of the present paper are closely related to
those introduced in \cite{CDYZh} for the 2D Aharonov--Bohm--Dirac
equation, however the setting here is quite different;
indeed, the geometry of the cosmic string spacetime
is more similar to geometry of flat cones. More new ingredients are needed in the proof, for example, in estimating the heat and Schr\"odinger kernels and obtaining Strichartz estimates. 
We recall that the wave diffraction and the dispersion of wave
and Schr\"odinger equations on flat cones were analyzed in
\cite{BFM, CT1, CT2, Ford, Zhang1}.
For the Dirac equation the situation is even harder, due to the rich algebraic structure of the equation. For the scalar Schr\"odinger operator on flat cones, a natural distinguished selfadjoint extension is available, namely, the Friedrichs extension; this choice is not available for Dirac, since the spectrum is unbounded from below.
Note that the choice of a good selfadjoint extenson is crucial in order to study dispersive properties of the equation, as it is evident from the results presented above.
An additional difficulty arises from the spin structure and the covariant derivative, which introduce variable coefficients on non flat manifolds and make the analysis harder.
To the best of our knowledge, the only available results in this direction are provided in  \cite{BCSZh,CD, CD2, CDM2,CDM}, where some local smoothing and Strichartz estimates with angular regularity loss are proved for Dirac equations on asymptotically flat manifolds, and for suitable spherically symmetric manifolds. Recently, in \cite{MZZ}, the authors studied Strichartz estimates for the Dirac equation on hyperbolic spaces. The Strichartz estimates proved in this paper present no loss of regularity, but the range of admissible pairs is influenced by the choice of the selfadjoint extension.

A brief outline of the paper. In Section \ref{sec:sa} we
classify all selfadjoint extensions of the Dirac operator on a cosmic string. In Section \ref{sec:hantra} we construct the necessary tools to represent the Dirac flow in this setting, and the actual construction of the propagator is carried out in Section \ref{sec:dirapro}, where an explicit representation of the kernel is obtained. The last two
Sections \ref{sec:dispest} and \ref{sec:stri} are dedicated to the proofs of the main results i.e.~dispersive and Strichartz estimates.

\section{Self-adjoint extensions }\label{sec:sa}
In this section we prove that the symmetric radial Dirac operators
$d_{k}$ defined in \eqref{eq:dk} are essentially self-adjoint for $|k|\geq1$, and allow a one--parameter family of selfadjoint extensions for $|k|<1$, which can be constructed as \emph{von Neumann extensions}.
Then we can establish the s.a.~extensions of the Dirac operator $\mathcal{D}_\sigma$.

Recalling the decomposition \eqref{L2Domc} and \eqref{varphi},
for any
$f=(\begin{smallmatrix} \phi \\ \psi \end{smallmatrix})
  \in \big[L^{2}(X)\big]^{2}$, we can
write
\begin{equation*}
\begin{split}
  \mathcal{D}_\sigma f&= \mathcal{D}_\sigma
  \sum_{k\in\Z}
  \begin{pmatrix}
    e^{-i \frac k\sigma\theta}\ \phi_{k}(r) \\
    e^{-i\frac k\sigma\theta}\ \psi_{k}(r)
  \end{pmatrix}
  =
  \sum_{k\in\Z}
  \begin{pmatrix}
    e^{-i\frac k\sigma \theta}(-\partial_r-\frac1{r}(\frac12-\frac k{\sigma }))\psi_{k}(r) \\
    e^{-i\frac k\sigma\theta}(\partial_r+\frac1r(\frac12+\frac k{\sigma }))\phi_{k}(r)
  \end{pmatrix}\\
  & =\sum_{k\in\Z}
e^{-i\frac k\sigma \theta} d_k \begin{pmatrix} \phi_k(r) \\ \psi_k(r) \end{pmatrix}.
  \end{split}
\end{equation*}
We note the following elementary facts:
\begin{enumerate}
  \item The radial operator
  $d_{k}$ with domain $D(d_{k})=\big[C_{c}^{\infty}((0,+\infty))\big]^{2}$
  is symmetric and densely defined;
  \item
  By the fact that the decomposition \eqref{L2Domc} is orthogonal, we have
  for the adjoint operator and its closure
  \begin{equation*}
    \mathcal{D}_\sigma^{*}=
    \bigoplus_{k\in \mathbb{Z}}d_{k}^{*}\otimes I_{2},
    \qquad
    \overline{\mathcal{D}}_\sigma=
    \bigoplus_{k\in \mathbb{Z}}\overline{d}_{k}\otimes I_{2}.
  \end{equation*}
  Hence, in order to establish the s.a.~extensions of $\mathcal{D}_\sigma$ it is enough to determine the s.a. extensions of $d_{k}$ for each $k$ and combine them in all possible ways.
\end{enumerate}
Now we consider the closure $\overline{d}_{k}$ and the adjoint $d^{*}_{k}$ respectively.\vspace{0.1cm}

\underline{\bf The closure $\overline{d}_{k}$}. Denote
\begin{align}\label{eq:pa}
\partial_\mu=\partial_r +\frac\mu r,\quad \partial_r=\frac{d}{dr},\quad \mu=\frac12+\frac k{\sigma }~or~\frac12-\frac k{\sigma }.
\end{align}
If $\phi\in C_{c}^{\infty}((0,+\infty))$, expanding the square and using integration by parts we get
\begin{equation*}
  \textstyle
  \int_{0}^{\infty}|\partial_{\mu}\phi|^{2}rdr=
  \int_0^\infty
  (|\partial_{r}\phi|^{2}+\frac{\mu^{2}}{r^{2}}|\phi|^{2})
  rdr.
\end{equation*}
Then for
$f=(\begin{smallmatrix}\phi \\ \psi\end{smallmatrix})\in \big[C_{c}^{\infty}((0,+\infty))\big]^{2}$, it gives the norm of the vector
\begin{equation*}
  \textstyle
  \|d_{k}f\|_{[L^{2}(rdr)]^2}^{2}=
  \|\phi'\|_{L^{2}(rdr)}^{2}+\|\psi'\|_{L^{2}(rdr)}^{2}+
  \|(\frac k{\sigma r}+\frac{1}{2r})\psi\|_{L^{2}(rdr)}^{2}+
  \|(\frac k{\sigma r}-\frac{1}{2r})\phi\|_{L^{2}(rdr)}^{2}.
\end{equation*}
Besides, the latter norm controls $\|f\|_{[L^{\infty}]^2}$:
\begin{equation*}
  \textstyle
  |\phi(b)|^{2}-|\phi(a)|^{2}=
  \int_{a}^{b}\partial_{r}|\phi(r)|^{2}dr\le
  \int_{a}^{b}(|\phi'|^{2}+\frac{|\phi|^{2}}{r^{2}})rdr
\end{equation*}
and hence if
$\phi\in C_{c}^{\infty}((0,+\infty))$
\begin{equation*}
  \|\phi\|_{L^{\infty}}\le\|\phi'\|_{L^{2}(rdr)}+
  \|\phi/r\|_{L^{2}(rdr)}.
\end{equation*}
If we take the series $f_{n}\in D(d_{k})$ such that both $f_{n}$ and
$d_{k}f_{n}$ converge in $[L^{2}(rdr)]^2$; since
$f_{n}\in \big[C_{c}^{\infty}\big]^2$, we have $f=\lim f_{n}$
is in $\big[C([0,+\infty))\big]^2$ and $f(0)=0$.
Therefore, we conclude that the domain of the closure $\overline{d}_{k}$
\begin{equation}\label{eq:domain}
  Y=
  D(\overline{d}_{k})=
  \big\{\phi\in L^{2}(rdr):
  \phi',\phi/r\in L^{2}(rdr),\
  \phi\in C([0,+\infty)),\
  \phi(0)=0\big\}^{2}
\end{equation}
and thus for all
$f=(\begin{smallmatrix}\phi \\ \psi\end{smallmatrix})
  \in D(\overline{d}_{k})$
\begin{equation*}
  \overline{d}_{k}f=\begin{pmatrix}0&-\partial_{\frac12-\frac k\sigma}\\ \partial_{\frac12+\frac k\sigma}&0\end{pmatrix}\begin{pmatrix}\phi\\ \psi\end{pmatrix}=
  \begin{pmatrix}
     -\partial_{\frac12-\frac k\sigma}\psi  \\
     \partial_{\frac12+\frac k\sigma}\phi
  \end{pmatrix}.
\end{equation*}
We observe that the functions in domain $Y$ are absolutely continuous so that
the distributional derivative coincides with the classical
derivative a.e.\vspace{0.1cm}

\underline{\bf The adjoint $d^{*}_{k}$}.
By definition
$f=(\begin{smallmatrix}\phi \\ \psi\end{smallmatrix})
  \in D(d^{*}_{k})$
iff there exists a $ g\in [L^{2}(rdr)]^{2}$ such that
\begin{equation*}
  \langle f,d_{k}u\rangle=
  \langle g,u\rangle
  \qquad
  \forall u\in (C_{c}^{\infty})^{2},
\end{equation*}
and we define $g=d_k^*f$. We can rewrite it as
\begin{equation*}
  f\in [L^{2}(rdr)]^{2}
  \quad\text{and}\quad
  d_{k}f\in [L^{2}(rdr)]^{2},
\end{equation*}
where $d_{k}f$ is computed in distribution sense. We conclude
\begin{equation*}
  D(d^{*}_{k})=
  \big\{f\in [L^{2}(rdr)]^{2}:d_{k}f\in [L^{2}(rdr)]^{2}\big\}
\end{equation*}
and the expression of $d^{*}_{k}$ is the same as $d_{k}$.\vspace{0.1cm}

To find all possible selfadjoint extensions of $\overline{d}_{k}$, first we have to make sure the operator $\overline{d}_{k}$ have equal deficiency indices.

\underline{\bf Deficiency indices of $\overline{d}_{k}$}.
Note that $\overline{d}_{k}$ is closed, densely defined and symmetric operator, thus von Neumann's theory works.
To extend the symmetric operator $d_k$ to a selfadjoint operator we need to compute the deficiency indices. Firstly we compute $\ker(d_{k}^{*}-i)$; it suffices to solve the system
\begin{equation}\label{eq:syseq1}
   -\partial_{\frac12-\frac k\sigma}\psi=i \phi,
  \qquad
  \partial_{\frac12+\frac k\sigma}\phi=i \psi,
\end{equation}
with $\phi,\psi\in L^{2}(rdr)$. This implies two coupling modified Bessel equations
\begin{equation}\label{eq:bessel1}
  \begin{cases}
  \big(\partial_{r}^{2}+\frac 1r \partial_{r}
    -(\frac1{2r}+\frac k{\sigma r})^2\big)\phi=\phi,\\
  \big(\partial_{r}^{2}+\frac 1r \partial_{r}
    -(\frac1{2r}-\frac k{\sigma r})^2\big)\psi=\psi.
  \end{cases}
\end{equation}
$I_{\mu},K_{\mu}$ are two linearly dependent solutions to this modified Bessel equations.
From \cite[10.25.3; 10.30.4]{DLMF}, one knows that
$I_{\mu}$ grows at infinity for all $\mu=|\frac12\pm\frac k\sigma|$ and is not in $L^{2}(rdr)$, so we drop it. Fortunately, \cite[10.25.3; 10.30.2]{DLMF} shows that $K_{\mu}$ decays exponentially as $r\to+\infty$, while near $r \sim 0$ we have
$K_{\mu}(r)\sim r^{-|\mu|}$
(indeed $K_{-\mu}=K_{\mu}$) so that
\begin{equation*}
  K_{\mu}\in L^{2}(rdr)
  \qquad\iff\qquad
  |\mu|<1.
\end{equation*}
Recall \eqref{eq:pa} and the standard recurrence relations for modified Bessel functions
\begin{equation}\label{eq:mB}
  \partial_{\mu}K_{\mu}=-K_{\mu-1},\quad \partial_{-\mu}K_{\mu}=-K_{\mu+1}
\end{equation}
which gives the component $\psi$.
Therefore we obtain such only one possible solution
\begin{equation*}
  \begin{pmatrix}
    \phi\\
     \psi
  \end{pmatrix}=
  \begin{pmatrix}
     K_{\frac k\sigma+\frac12}  \\
     iK_{\frac k\sigma-\frac12}
  \end{pmatrix}
\end{equation*}
and of course its multiples.
This solution is in the domain of $d^{*}_{k}$ iff
$K_{\mu}\in L^{2}(rdr)$ and
$\partial_{\mu}\phi=i\psi=-K_{\mu-1}\in L^{2}(rdr)$.
Consequently, we must have $K_{\frac k\sigma+\frac12},K_{\frac k\sigma-\frac12}\in L^{2}(rdr)$ and that $0<\sigma\leq 1$, then we have
\begin{equation*}
 \big |\frac k\sigma+\frac12\big|<1,\quad \big|\frac k\sigma-\frac12\big|<1
  \quad\Leftrightarrow\quad
  k=0.
\end{equation*}
We conclude:
\begin{itemize}
  \item if $k\neq0$ then
  $\ker(d^{*}_{k}-i)=\{0\}$
  \item
  $\ker(d^{*}_{0}-i)=
    \Big[\big(\begin{smallmatrix} K_{1/2}  \\ iK_{1/2}
    \end{smallmatrix}\big)\Big]$,
  with dimension $n_{+}=1$.
\end{itemize}
The analysis for $\ker(d^{*}_{k}+i)$ is similar to above.
We now solve the following system
\begin{equation}\label{eq:syseq2}
   -\partial_{\frac12-\frac k\sigma}\psi=-i \phi,
  \qquad
   \partial_{\frac12+\frac k\sigma}\phi=-i \psi,
\end{equation}
which gives the same second order modified Bessel equations for $\phi,\psi$. However, by the relations \eqref{eq:syseq2} the solutions are given by
\begin{equation*}
  \begin{pmatrix}
    \phi\\
     \psi
  \end{pmatrix}=
  \begin{pmatrix}
     K_{\frac{k}{\sigma}+\frac12}  \\
     -iK_{\frac{k}{\sigma}-\frac12}
  \end{pmatrix}.
\end{equation*}
Follow the same line of computations we obtain
\begin{itemize}
  \item if $k\neq0$ then
  $\ker(d^{*}_{k}+i)=\{0\}$
  \item
  $\ker(d^{*}_{0}+i)=
    \Big[\big(\begin{smallmatrix} K_{1/2}  \\ -iK_{1/2}
    \end{smallmatrix}\big)\Big]$,
  with dimension $n_{-}=1$.
\end{itemize}
Summing up, we get the deficiency indices $(n_+,n_-)=(1,1)$ for $k=0$. This means the operator $d_k$ allows the s.a. extensions.

\underline{\bf S.a. extensions} It is now to straightforward give a selfadjoint extension and determine the domain via the von Neumann theory:
\begin{enumerate}
  \item For $k\neq0$, the deficiency indices $(n_+,n_-)=(0,0)$, it shows
  $\overline{d}_{k}$ is selfadjoint, with domain $Y$ in \eqref{eq:domain},
  \item For $k=0$,
  \begin{equation*}
    D(d^{*}_{0})=
    Y
  \oplus
    \left[g_{+} \right]
    \oplus
    \left[g_{-}\right]
    \quad\text{where}\quad
    g_{+}=
    \begin{pmatrix} K_{\frac12}  \\ iK_{\frac12} \end{pmatrix},
    \quad
    g_{-}=
    \begin{pmatrix} K_{\frac12} \\ -iK_{\frac12} \end{pmatrix}
  \end{equation*}
  and for any $f\in D(\overline{d}_{0})$,
  $c_{1},c_{2}\in \mathbb{C}$,
  \begin{equation*}
    d^{*}_{0}(f+c_{1}g_{+}+c_{2}g_{-})=
    \overline{d}_{0}f+ic_{1}g_{+}-ic_{2}g_{-}.
  \end{equation*}
\end{enumerate}
All selfadjoint extensions $d_{0}^{\gamma}$ of $d_{0}$ are in one to one correspondence with unitary operators from $N_{+}$ onto $N_{-}$ and are determined by boundary conditions at the origin. We discuss the boundary conditions in next section. From the deficiency spaces we can derive such a unitary operator from $N_+$ to $N_-$ are multiplication by a number $e^{2i \gamma}$ with $\gamma$ real.
  That is
  \begin{equation*}
    D(d^{\gamma}_{0})=
    \{f+c g_{+}+ce^{2i \gamma}g_{-}:f\in Y,\
      c\in \mathbb{C}\}=
      Y \oplus[g_{+}+e^{2 i \gamma}g_{-}]
  \end{equation*}
  and for all $f\in Y$, $c\in \mathbb{C}$
  \begin{equation*}
    d_{0}^{\gamma}(f+c g_{+}+ce^{2i \gamma}g_{-})=
    \overline{d}_{0}f+ic g_{+}-ice^{2i \gamma}g_{-}.
  \end{equation*}
  In order to get a more concise representation, we rewritten
  the last two formulas as:
  \begin{equation*}
    c g_{+}+ce^{2i \gamma}g_{-}=
    ce^{i \gamma}
    \begin{pmatrix}
      e^{-i \gamma} K_{\frac12}+e^{i \gamma}K_{\frac12}
      \\
      ie^{-i \gamma}K_{\frac12} -ie^{i \gamma}K_{\frac12}
    \end{pmatrix}=
    2ce^{i \gamma}
    \begin{pmatrix}
      \cos \gamma\cdot K_{\frac12}
      \\
      \sin \gamma\cdot K_{\frac12}
    \end{pmatrix}.
  \end{equation*}
  \begin{equation*}
    ic g_{+}-ice^{2i \gamma}g_{-}=
    ic e^{i \gamma}
    \begin{pmatrix}
      e^{-i \gamma} K_{\frac12}-e^{i \gamma}K_{\frac12}
      \\
      ie^{-i \gamma}K_{\frac12} +ie^{i \gamma}K_{\frac12}
    \end{pmatrix}=
    2ce^{i \gamma}
    \begin{pmatrix}
      \sin \gamma\cdot K_{\frac12}
      \\
      -\cos \gamma\cdot K_{\frac12}
    \end{pmatrix}.
  \end{equation*}
  Thus we have
  \begin{equation*}
    D(d^{\gamma}_{0})=
    \{f+c
    \left(
    \begin{smallmatrix}
      \cos \gamma \cdot K_{\frac12} \\
      \sin \gamma \cdot K_{\frac12}
    \end{smallmatrix}
    \right)
    : f\in Y,\ c\in \mathbb{C}\}=
    Y \oplus
    \left[\left(
    \begin{smallmatrix}
      \cos \gamma \cdot K_{\frac12} \\
      \sin \gamma \cdot K_{\frac12}
    \end{smallmatrix}
    \right)\right]
  \end{equation*}
  and
  \begin{equation}\label{eq:d0g}
    d_{0}^{\gamma}
    \left(
    f+c
    \left(
    \begin{smallmatrix}
      \cos \gamma \cdot K_{\frac12} \\
      \sin \gamma \cdot K_{\frac12}
    \end{smallmatrix}
    \right)
    \right)=
  \overline{d}_{0}f+c
    \left(
    \begin{smallmatrix}
      \sin \gamma \cdot K_{\frac12} \\
      -\cos \gamma \cdot K_{\frac12}
    \end{smallmatrix}
    \right).
  \end{equation}

In what follows we still write
 \begin{align}\label{eq:dkg}
 d^\gamma_k=\overline{d}_k,\qquad D(d^\gamma_k)=Y,\qquad k\neq0.
 \end{align}
Finally we shall give the representation of all extensions of $\D_{\sigma}$ initially defined on $C^\infty_0(X\setminus0)$. Recall that $Y$ defined in \eqref{eq:domain} is the domain of the closures $\overline{d}_k$. For any $\gamma\in\R$ and $k$, we define the domain as
 \begin{equation}
 \begin{aligned}
 Y^\gamma&=D(d^\gamma_k)=D(\overline{d}_k)\oplus\left[g_{+} \right]
    \oplus
    \left[g_{-}\right]\\&=\left[\left(
    \begin{smallmatrix}
      \phi \\
      \psi
    \end{smallmatrix}
    \right)\right] \oplus
    \left[\left(
    \begin{smallmatrix}\cos\gamma\cdot K_{\frac12}\\\sin\gamma\cdot K_{\frac12}\end{smallmatrix}\right)\right],\quad \left(
    \begin{smallmatrix}
      \phi \\
      \psi
    \end{smallmatrix}
    \right)\in Y.
 \end{aligned}
 \end{equation}
Hence, the action of $d^\gamma_0$ can be written simply as follows:
 \begin{align*}
 d^\gamma_0\begin{pmatrix}
     \phi_{0}  \\
     \psi_{0}
  \end{pmatrix}
  =c\begin{pmatrix}
      \sin \gamma \cdot K_{\frac12} \\
      -\cos \gamma \cdot K_{\frac12}
    \end{pmatrix}+
  \begin{pmatrix}-\partial_{\frac12}\tilde{\psi}_0(r)\\ \partial_{\frac12}\tilde{\phi}_0(r)\end{pmatrix}, \quad \forall \;\begin{pmatrix}\phi_0\\ \psi_0\end{pmatrix}=c\begin{pmatrix}
      \cos \gamma \cdot K_{\frac12} \\
      \sin \gamma \cdot K_{\frac12}
    \end{pmatrix}+\begin{pmatrix}
     \tilde{\phi}_{0}  \\
     \tilde{\psi}_{0}
  \end{pmatrix}\in Y^\gamma.
 \end{align*}

Then we obtain the following result:
\begin{theorem}[Selfadjoint extensions for Dirac]\label{th:sa}
Any selfadjoint extension of $\D_{\sigma}$ is of the form $\D_{\sigma,\gamma}$, where $\gamma\in [0,2\pi)$,
 \begin{align*}
 D(\D_{\sigma,\gamma})=(Y^\gamma\otimes h_0)\oplus_{k\neq0}(Y\otimes h_k)
 \end{align*}
and
 \begin{align*}
 \D_{\sigma,\gamma}=(d^\gamma_0\otimes I_2)\oplus_{k\neq0}(\overline{d}_k\otimes I_2).
 \end{align*}
Any $f\in D(\D_{\sigma,\gamma})$ can be written as
 \begin{align*}
 f=\sum_{k\in\Z}\begin{pmatrix}\phi_k(r)e^{-i\frac{k}{\sigma}\theta}\\
 \psi_k(r)e^{-i\frac{k}{\sigma}\theta}\end{pmatrix}
 \end{align*}
with $\left(\begin{smallmatrix}\phi_k(r)\\ \psi_k(r)\end{smallmatrix}\right)\in Y$ for $k\neq0$ while $\left(\begin{smallmatrix}\phi_0(r)\\ \psi_0(r)\end{smallmatrix}\right)\in Y^\gamma$ for $k=0$, i.e., for some $c\in \C$ and $\left(\begin{smallmatrix}\tilde{\phi}_0\\ \tilde{\psi}_0\end{smallmatrix}\right)\in Y$,
\begin{align*}
\begin{pmatrix}\phi_0(r)\\ \psi_0(r)\end{pmatrix}=c\begin{pmatrix}\cos\gamma\cdot K_{\frac12}(r)\\ \sin\gamma\cdot K_{\frac12}(r)\end{pmatrix}+\begin{pmatrix}\tilde{\phi}_0(r)\\ \tilde{\psi}_0(r)\end{pmatrix}.
\end{align*}
The action of $\D_{\sigma,\gamma}$ on such $f$ is given by
 \begin{align}
 \D_{\sigma,\gamma}f=c
    \begin{pmatrix}
      \sin \gamma \cdot K_{\frac12}(r) \\
      -\cos \gamma \cdot K_{\frac12}(r)
    \end{pmatrix}
    +\sum_{k\neq0}\begin{pmatrix}-e^{-i\frac{k}{\sigma}\theta} \partial_{\frac12- \frac{k}{\sigma}}\psi_k(r)\\ e^{-i\frac{k}{\sigma}\theta} \partial_{\frac12+\frac{k}{\sigma}}\phi_k(r)\end{pmatrix}.
 \end{align}
 \end{theorem}

\begin{remark}\label{rem:H}
Under the decomposition \eqref{L2Domc}, regardless of the self-adjoint extension, we informally obtain
 \begin{align*}
 \D_\sigma^2 f(x)&=\sum_{k\in\Z}\D^2_{\sigma}\begin{pmatrix}e^{-i\frac k\sigma \theta}&0\\0&e^{-i\frac k\sigma \theta}\end{pmatrix}\begin{pmatrix}f^+_k(r)\\f^-_k(r)\end{pmatrix}\\
 &=\sum_{k\in\Z}\begin{pmatrix}0&-(\partial_r+\frac1{2r})+\frac 1rD_{\mathbb{S}_\sigma^1}\\(\partial_r+\frac1{2r})+\frac 1rD_{\mathbb{S}_\sigma^1}&0
 \end{pmatrix}^2
 \begin{pmatrix}f^+_k(r)e^{-i\frac k\sigma \theta}\\f^-_k(r)e^{-i\frac k\sigma \theta}\end{pmatrix}\\
 \end{align*}
and by \eqref{eq:a-eigen} we get
 \begin{align*}
 \D_\sigma^2 f(x)&=\sum_{k\in\Z}\begin{pmatrix}e^{-i\frac k\sigma \theta}&0\\0&e^{-i\frac k\sigma \theta}\end{pmatrix}
 \begin{pmatrix}0&-(\partial_r+\frac1{2r})+\frac{k}{\sigma r}\\
 (\partial_r+\frac1{2r})+\frac{k}{\sigma r}&0
 \end{pmatrix}^2\begin{pmatrix}f^+_k(r)\\f^-_k(r)\end{pmatrix}\\
 &=\sum_{k\in\Z}\begin{pmatrix}[-\partial_r^2-\frac 1r\partial_r+\frac1{r^2}(\frac1{2}+\frac k{\sigma})^2]f^+_k(r)e^{-i\frac k\sigma \theta}\\ [-\partial_r^2-\frac 1r\partial_r+\frac1{r^2}(\frac1{2}-\frac k{\sigma})^2]f^-_k(r)e^{-i\frac k\sigma \theta}\end{pmatrix}.
 \end{align*}
 We define
 \begin{align}
L_k:=\begin{pmatrix}H_{\nu_+}&0\\0&H_{\nu_-}  \end{pmatrix}=\begin{pmatrix}0&-(\partial_r+\frac1{2r})+\frac{k}{\sigma r}\\
 (\partial_r+\frac1{2r})+\frac{k}{\sigma r}&0
 \end{pmatrix}^2
\end{align}
where $\nu_\pm=|\frac1{2}\pm\frac k{\sigma}|$ and $H_{\nu_\pm}=-\partial_r^2-\frac 1r\partial_r+\frac1{r^2}(\frac1{2}\pm\frac k{\sigma})^2$. Corresponding, we define
 \begin{align}\label{eq:square}
 H:=\begin{pmatrix}H^+_\sigma&0\\ 0&H^-_\sigma\end{pmatrix}=\bigoplus_{k\in\Z}L_k\otimes I_2.
 \end{align}
where  $H^\pm_\sigma:=-\partial_r^2-\frac 1r\partial_r+\frac1{r^2}(\frac12\pm i\partial_\theta)^2$ is a positive operator on the space $X=(0,+\infty)\times\mathbb{S}^1_{\sigma}$, initially defined on $C_{c}^{\infty}(X)$
(see \cite{Zhang1}); This is a second order differential operator whose spectrum is bounded from below, so we have the Friedrichs extension denoted by $H^\pm_\sigma$ again.
More precisely, similar to above, we extend the domain of  $L_k$ to
 \begin{equation}
 \begin{split}
 D(L_k) =\{\phi\in L^2(rdr): &\phi', \phi'', \phi'/r, \phi/r^2\in L^2(rdr),\\ & \phi\in C^1([0,\infty))), \phi(0)=\phi'(0)=0 \}^2.
 \end{split}
 \end{equation}
 such that the radial operator $L_k$ is self-adjoint. Hence we have the Friedrichs extension of $H$. Then, regardless of the self-adjoint extension, we informally obtain
 \begin{align}\label{equ:D2=H}
 \D_\sigma^2 f(x)&=\sum_{k\in\Z}\begin{pmatrix}e^{-i\frac k\sigma \theta}&0\\0&e^{-i\frac k\sigma \theta}\end{pmatrix}
L_k\begin{pmatrix}f^+_k(r)\\f^-_k(r)\end{pmatrix}= H f(x).
 \end{align}

%
%

\end{remark}

\section{The relativistic Hankel transform 
and the Dirac propagator}
\label{sec:hantra}
In this section, we construct the Dirac propagator by deriving the generalized eigenfunctions of $d^\gamma_k$, and
defining the relativistic Hankel transform.
This is closed to the paper \cite{CDYZh} about Dirac equation with the Aharonov-Bohm potential.

\subsection{The generalized eigenfunctions}
In this subsection, we begin with the radial eigenvalue problem by the composition \eqref{L2Domc}
 \begin{align}\label{eq:eg}
 d^\gamma_k\Psi_k(\rho r)=\begin{pmatrix}0&-(\partial_r+\frac1{2r})+\frac{k}{\sigma r}\\
 (\partial_r+\frac1{2r})+\frac{k}{\sigma r}&0
 \end{pmatrix}\begin{pmatrix}\psi^1_k(\rho r)\\ \psi_k(\rho r)\end{pmatrix}=\rho \begin{pmatrix}\psi^1_k(\rho r)\\ \psi_k(\rho r)\end{pmatrix},\quad \rho\in\R.
 \end{align}
As well known the weak solutions of the eigenvalue equation \label{eq:eg} are the generalized eigenfunctions of $d^\gamma_k$. Taking the domain $D(d^\gamma_k)$ into account, one should not ignore the boundary conditions. For $k\neq0$, the boundary condition is $\Psi_k(0)=0$; for $k=0$, we need to care about the singular part.

By scaling, it is sufficient to verify \eqref{eq:eg} for $\rho=1$ due to the scaling properties of $d^\gamma_k$. Recalling the definition of $d_k$ in \eqref{eq:dk}, if $\rho=1$ then \eqref{eq:eg}
which reduces to the Bessel equations
 \begin{align}
 \frac{d^2}{dr^2}\psi^1_k(r)+\frac1r \frac{d}{dr}\psi^1_k(r)+\Big(1-\frac1{r^2}\big(\frac12+\frac{k}{\sigma}\big)^2\Big)\psi^1_k(r)=0,\\
 \frac{d^2}{dr^2}\psi^2_k(r)+\frac1r \frac{d}{dr}\psi^2_k(r)+\Big(1-\frac1{r^2}\big(\frac12-\frac{k}{\sigma}\big)^2\Big)\psi^2_k(r)=0.
 \end{align}
Then we get the solutions of \eqref{eq:eg} for $\rho=1$
 \begin{align}\label{eq:so}
 \begin{pmatrix}\psi^1_k( r)\\ \psi^2_k( r)\end{pmatrix}=\begin{pmatrix}J_{\frac{k}{\sigma}+\frac12}(r)\\ J_{\frac{k}{\sigma}-\frac12}(r)\end{pmatrix}\; \text{or}\;\begin{pmatrix}J_{-(\frac{k}{\sigma}+\frac12)}(r)\\ -J_{-(\frac{k}{\sigma}-\frac12)}(r)\end{pmatrix}.
 \end{align}
As far as the generalized eigenfunctions are concerned, by the asymptotic behaviour of Bessel functions
as $x\to0$
$$
J_\nu(x) = \frac{x^\nu}{2^\nu \Gamma(1+\nu)}+O(x^{\nu+2}),
$$
the square integrability at the origin restricts $\frac{k}{\sigma} \not\in(-\frac12,\frac12)$, thus for $k\neq0$ we take the generalized eigenfunctions as follows
 \begin{align}\label{egk}
 \begin{cases}
 \psi^1_k(r)=J_{|\frac{k}{\sigma}+\frac12|}(r),\; \psi^2_k(r)=J_{|\frac{k}{\sigma}-\frac12|}(r)& {\rm if}\; k\geq1,\\
 \psi^1_k(r)=J_{|\frac{k}{\sigma}+\frac12|}(r),\; \psi^2_k(r)=-J_{|\frac{k}{\sigma}-\frac12|}(r)& {\rm if}\; k\leq-1.
 \end{cases}
 \end{align}

Next we consider the case $k=0$. The choices in \eqref{eq:so} for $k=0$ both lead to solutions that are square integrable near the origin while one component is more singular than the other. Therefore, Theorem \ref{th:sa} states that the choice of the selfadjoint extension will determine the choice of the generalized eigenfunction in the case $k=0$.
The eigenfunction admits the form
 \begin{equation}\label{gef0}
 \begin{pmatrix}
 \psi^1_0(r)\\
 \psi^2_0(r)\end{pmatrix}= c_1 \begin{pmatrix}J_{\frac12}(r)\\J_{-\frac12}(r)\end{pmatrix}+
 c_2\begin{pmatrix}J_{-\frac12}(r)\\ - J_{\frac12}(r)\end{pmatrix},
 \end{equation}
where $c_1, c_2\in \C$.

In view of the radial Hamiltonian
 \begin{align*}
 d^\gamma_0=\begin{pmatrix}0&-\frac{d}{dr}-\frac1{2r}\\ \frac{d}{dr}+\frac1{2r}&0 \end{pmatrix}
 \end{align*}
is symmetric, $\langle d^\gamma_0\chi(r),\varphi(r) \rangle_{L^2(rdr)}=\langle\chi(r),d^\gamma_0\varphi(r)\rangle_{L^2(rdr)}$ for all $\varphi, \chi\in D(d^\gamma_0)$.
Then, for $\varphi(r)=\left(
\begin{smallmatrix} \varphi_1(r) \\ \varphi_2(r)
\end{smallmatrix}
\right)$
and
$\chi(r)=\left(\begin{smallmatrix} \chi_1(r) \\ \chi_2(r)
\end{smallmatrix} \right)$ it can be easily shown as

 \begin{align*}
 \langle &d^\gamma_0\chi(r),\varphi(r) \rangle_{L^2(rdr)}=\int^\infty_0\begin{pmatrix}0&-\frac{d}{dr}-\frac1{2r}\\ \frac{d}{dr}+\frac1{2r}&0\end{pmatrix} \begin{pmatrix}\chi_1(r)\\ \chi_2(r)\end{pmatrix}\cdot \begin{pmatrix}\overline{\varphi_1}(r)\\ \overline{\varphi_2}(r)\end{pmatrix} rdr\\
 &=(\chi_1\overline{\varphi_2}r-\chi_2\overline{\varphi_1}r)|^\infty_0+ \int^\infty_0 \begin{pmatrix}\chi_1(r)\\ \chi_2(r)\end{pmatrix}\cdot \begin{pmatrix}0&-\frac{d}{dr}-\frac1{2r}\\ \frac{d}{dr}+\frac1{2r}&0\end{pmatrix} \begin{pmatrix}\overline{\varphi_1}(r)\\ \overline{\varphi_2}(r)\end{pmatrix} rdr
 \end{align*}
Therefore the symmetric property of the Dirac operator $d^\gamma_0$ requires certain boundary conditions
 \begin{align}\label{eq:bd}
 (\chi_1\overline{\varphi_2}r-\chi_2\overline{\varphi_1}r)|^\infty_0=0
 \end{align}
i.e.
 \begin{align}
 \lim_{r\rightarrow0}r(\chi_1\overline{\varphi_2}-\chi_2\overline{\varphi_1}) =\lim_{r\rightarrow\infty}r(\chi_1\overline{\varphi_2} -\chi_2\overline{\varphi_1})=0.
 \end{align}
Apply this to the spinors
$\varphi(r)=\left(
\begin{smallmatrix}
  \cos \gamma \cdot K_{\frac12} \\
  \sin \gamma \cdot K_{\frac12}
\end{smallmatrix}
\right)$
and
$\chi(r)=\left(
\begin{smallmatrix}
  \psi^1_0(r) \\
     \psi^2_0(r)
\end{smallmatrix}
\right)$,
which belong to $D(d^\gamma_0)$.
By the decay properties of Bessel function $J_{\nu}$, we can get
$$c_1=\tilde{c}\sin\gamma,\quad c_2=\tilde{c}\cos\gamma.$$
Then the boundary conditions \eqref{eq:bd} lead to the following
 \begin{align}\label{eq:s0}
 \Psi_0(r)=\begin{pmatrix}\psi^1_0(r)\\ \psi^2_0(r)\end{pmatrix}
 =\tilde{c}\begin{pmatrix}\sin\gamma J_{\frac12}(r)+\cos\gamma J_{-\frac12}(r)\\ \sin\gamma J_{-\frac12}(r)-\cos\gamma J_{\frac12}(r)\end{pmatrix}.
 \end{align}

\begin{remark}\rm
  For negative energies the computation is completely
  analogous. The negative eigenfunctions be represented by
  \begin{equation}\label{negen}
  \left(\begin{matrix}\phi_k^1(\rho r)\\ \phi_k^2(\rho r)\end{matrix}\right)=
  \left(\begin{matrix}\psi_k^1(|\rho| r)\\ -\psi_k^2(|\rho| r)\end{matrix}\right),\qquad \rho<0.
  \end{equation}
\end{remark}

\subsection{The relativistic Hankel transform}\label{sec:hank}

Firstly, from \eqref{egk} and \eqref{eq:s0} we choose the normalized, complete and orthonormal  eigenfunctions in the generalized sense: given $\sigma\in (0,1]$
 \begin{itemize}
 \item for $k\geq1$,
 \begin{align}\label{Psik+}
 \Psi_k(r)=\frac1{\sqrt{2}}\begin{pmatrix}J_{\frac{k}{\sigma}+\frac12}(r)\\ J_{\frac{k}{\sigma}-\frac12}(r)\end{pmatrix}
 \end{align}
 \item for $k\leq-1$,
 \begin{align}\label{Psik-}
  \Psi_k(r)=\frac1{\sqrt{2}}\begin{pmatrix}J_{-(\frac{k}{\sigma}+\frac12)}(r)\\ -J_{-(\frac{k}{\sigma}-\frac12)}(r)\end{pmatrix}
 \end{align}
 \item for $k=0$,
 \begin{align}\label{Psi0}
 \Psi_0(r)=\frac1{\sqrt{2}}\begin{pmatrix}\sin\gamma J_{\frac12}(r)+\cos\gamma J_{-\frac12}(r)\\ \sin\gamma J_{-\frac12}(r)-\cos\gamma J_{\frac12}(r)\end{pmatrix}.
 \end{align}
 \end{itemize}
 Keeping into account the case of
 negatives energies \eqref{negen}, we define the relativistic Hankel transform as a sequence of operators
 \begin{align}
 \mathcal{P}_k: [L^2(\R^+,rdr)]^2\rightarrow L^2(\R,|\rho|d\rho)
 \end{align}
acting on spinors $f(r)=\left(\begin{smallmatrix}f_1(r)\\f_2(r)\end{smallmatrix} \right)$, $r\in(0,\infty)$, as
 \begin{align}\label{eq:Pk}
 \Pk f(\rho)=\begin{cases} \int^\infty_0\Psi_k(\rho r)^T\cdot f(r)rdr &\text{if}\; \rho>0\\
 \int^\infty_0\Psi_k(|\rho| r)^T\cdot f(r)rdr &\text{if}\; \rho<0. \end{cases}
 \end{align}
Recall the Hankel transforms $\mathcal{H}_\nu$ of order $\nu$, defined as
 \begin{align}\label{Hankel}
 \mathcal{H}_\nu \phi(\rho)=\int^\infty_0J_{\nu}(r\rho )\phi(r)rdr, \quad \rho>0.
 \end{align}
Then we can write $\Pk$ as a combination of Hankel transforms:
\begin{itemize}
\item for $k\geq1$, with $\nu=\frac{k}{\sigma}+\frac12$,
 \begin{align}\label{pk+}
 \Pk f(\rho)=2^{-\frac12}\begin{cases}\mathcal{H}_\nu f_1(\rho)+\mathcal{H}_{\nu-1}f_2(\rho)  &\text{if}\; \rho>0,\\
 \mathcal{H}_\nu f_1(|\rho|)-\mathcal{H}_{\nu-1}f_2(|\rho|)  &\text{if}\; \rho<0, \end{cases}
 \end{align}

\item for $k\leq-1$, with $\nu=-(\frac{k}{\sigma}+\frac12)$,
 \begin{align}\label{pk-}
 \Pk f(\rho)=2^{-\frac12}\begin{cases}\mathcal{H}_\nu f_1(\rho)-\mathcal{H}_{\nu+1}f_2(\rho)  &\text{if}\; \rho>0,\\
 \mathcal{H}_\nu f_1(|\rho|)+\mathcal{H}_{\nu+1}f_2(|\rho|)  &\text{if}\; \rho<0, \end{cases}
 \end{align}

\item for $k=0$,
 \begin{align}\label{pk0}
 \mathcal{P}_0 f(\rho)=2^{-\frac12} \begin{cases}\sin\gamma[\mathcal{H}_{\frac12}f_1(\rho) +\mathcal{H}_{-\frac12}f_2(\rho)]+\cos\gamma[\mathcal{H}_{-\frac12}f_1(\rho) -\mathcal{H}_{\frac12}f_2(\rho)] &\text{if}\; \rho>0,\\
 \sin\gamma[\mathcal{H}_{\frac12}f_1(|\rho|) -\mathcal{H}_{-\frac12}f_2(|\rho|)]+\cos\gamma[\mathcal{H}_{-\frac12}f_1(|\rho|) +\mathcal{H}_{\frac12}f_2(|\rho|)] &\text{if}\; \rho<0.  \end{cases}
 \end{align}
\end{itemize}
Since the Hankel transform $\mathcal{H}_\nu$ is a unitary operator on $L^2(\R^+,rdr)$ and satisfies
 \begin{align}\label{eq:Hk}
 \mathcal{H}^2_\nu=I,\quad \mathcal{H}_\nu\partial_{\nu+1}=r\mathcal{H}_{\nu+1},\quad \mathcal{H}_{\nu+1}\partial_{-\nu}=-r\mathcal{H}_{\nu},
 \end{align}
it shows $\mathcal{P}_k$ is a well defined operator for any $k$.

Now we verify that the operator $\mathcal{P}_k$ diagonalizes $d_k$. Using \eqref{eq:Hk} we obtain
\begin{itemize}
\item for $k\geq1$, with $\nu=\frac{k}{\sigma}+\frac12$,
 \begin{align*}
 \Pk \overline{d}_k\begin{pmatrix}\phi\\ \psi\end{pmatrix}&=\Pk \begin{pmatrix}-\partial_{-(\nu-1)}\psi\\ \partial_{\nu}\phi\end{pmatrix}(\rho)\\
 &=2^{-\frac12}\begin{cases}\mathcal{H}_{\nu}(-\partial_{-(\nu-1)}\psi)(\rho) +\mathcal{H}_{\nu-1}(\partial_\nu \phi)(\rho), &\rho>0,\\
 \mathcal{H}_{\nu}(-\partial_{-(\nu-1)}\psi(|\rho|)) -\mathcal{H}_{\nu-1}(\partial_\nu \phi(|\rho|)), &\rho<0,\end{cases}\\
 &=2^{-\frac12}\begin{cases}
 \rho(\mathcal{H}_{\nu}\phi(\rho)+\mathcal{H}_{\nu-1}\psi(\rho)), \quad &\rho>0,\\
 |\rho|(-\mathcal{H}_{\nu}\phi(|\rho|)+\mathcal{H}_{\nu-1}\psi(|\rho|)), \quad &\rho<0,
 \end{cases}\\
 &=\rho \Pk\begin{pmatrix}\phi\\ \psi\end{pmatrix}.
 \end{align*}
\item for $k\leq-1$, with $\nu=-(\frac{k}{\sigma}+\frac12)$,
 \begin{align*}
 \Pk \overline{d}_k\begin{pmatrix}\phi\\ \psi\end{pmatrix}&=\Pk \begin{pmatrix}-\partial_{\nu+1}\psi\\ \partial_{-\nu}\phi\end{pmatrix}(\rho) \\ &=2^{-\frac12}\begin{cases}\mathcal{H}_{\nu}(-\partial_{\nu+1}\psi) (\rho) -\mathcal{H}_{\nu+1}(\partial_{-\nu} \phi)(\rho), &\rho>0,\\
 \mathcal{H}_{\nu}(-\partial_{\nu+1}\psi(|\rho|)) -\mathcal{H}_{\nu+1}(\partial_{-\nu} \phi(|\rho|)), &\rho<0,\end{cases}\\
 &=2^{-\frac12}\begin{cases}
 \rho(\mathcal{H}_{\nu}\phi(\rho)-\mathcal{H}_{\nu+1}\psi(\rho)), \quad &\rho>0,\\
 |\rho|(\mathcal{H}_{\nu}\phi(|\rho|)+\mathcal{H}_{\nu+1}\psi(|\rho|)), \quad &\rho<0,
 \end{cases}\\
 &=\rho \Pk\begin{pmatrix}\phi\\ \psi\end{pmatrix}.
 \end{align*}
\item for $k=0$, denote
 \begin{align*}
 v_{\gamma}=\begin{pmatrix}\cos\gamma\cdot K_{\frac12}(r)\\ \sin\gamma\cdot K_{\frac12}(r)\end{pmatrix},\quad d^\gamma_0 v_\gamma=\begin{pmatrix}\sin\gamma\cdot K_{\frac12}(r)\\ -\cos\gamma\cdot K_{\frac12}(r)\end{pmatrix},
 \end{align*}
 using the identity
 $$\mathcal{H}_\nu K_{\nu}(\rho)=\frac{\rho^\nu}{1+\rho^2},\quad \nu>-1,$$
 we have
 \begin{align*}
 \mathcal{P}_0 v_\gamma(\rho)=\frac{2^{-\frac12}}{1+\rho^2}\begin{cases}
 \sin\gamma(\rho^{\frac12}\cos\gamma+\rho^{-\frac12}\sin\gamma)+\cos\gamma( \rho^{-\frac12}\cos\gamma-\rho^{\frac12}\sin\gamma), &\rho>0, \\
 \sin\gamma(|\rho|^{\frac12}\cos\gamma-|\rho|^{-\frac12}\sin\gamma)+\cos\gamma( |\rho|^{-\frac12}\cos\gamma+|\rho|^{\frac12}\sin\gamma), &\rho<0,
 \end{cases}
 \end{align*}
and
 \begin{align*}
 \mathcal{P}_0 d^\gamma_0 v_{\gamma}(\rho)=\frac{2^{-\frac12}}{1+\rho^2}\begin{cases}
 \sin\gamma(\rho^{\frac12}\sin\gamma-\rho^{-\frac12}\cos\gamma)+\cos\gamma( \rho^{-\frac12}\sin\gamma+\rho^{\frac12}\cos\gamma), &\rho>0, \\
 \sin\gamma(|\rho|^{\frac12}\sin\gamma+|\rho|^{-\frac12}\cos\gamma)+\cos\gamma( |\rho|^{-\frac12}\sin\gamma-|\rho|^{\frac12}\cos\gamma), &\rho<0.
 \end{cases}
 \end{align*}
Summing up we get
 \begin{align*}
 \mathcal{P}_0 d^\gamma_0 v_{\gamma}(\rho)-\rho \mathcal{P}_0v_\gamma(\rho)=2^{-\frac12}\begin{cases}0, &\text{if}\; \rho>0,\\ 2\sin\gamma\cos\gamma |\rho|^{-\frac12}, &\text{if}\; \rho<0.\end{cases}
 \end{align*}
It's easy to see that
 \begin{align}\label{eq:rH0}
 \mathcal{P}_0d^\gamma_0 v_{\gamma}=\rho\mathcal{P}_0v_\gamma \qquad \text{holds on $Y^\gamma$\; iff} \qquad \sin\gamma\cos\gamma=0.
 \end{align}
\end{itemize}
Then in what follows we have to assume that $\sin \gamma\cos \gamma=0$.
Note that if $\cos \gamma=0$ then $\mathcal{P}_{0}$
takes the form \eqref{pk+} with
$\nu=\frac12$, while if $\sin \gamma=0$ then
$\mathcal{P}_{0}$ has the form \eqref{pk-} with $\nu=-\frac12$.

Finally, it remains to prove $\Pk$ is invertible. Given $\phi\in L^2(\R, |\rho|d\rho)$, write $\phi_+(\rho)=\phi(\rho)$ and $\phi_{-}(\rho)=\phi(-\rho)$ for $\rho>0$. Then we can write
 \begin{align}\label{ipk+}
 \Pk^{-1}\phi(r)=2^{-\frac12} \begin{pmatrix}
 \mathcal{H}_{\nu}&\mathcal{H}_{\nu}\\
 \mathcal{H}_{\nu-1}&-\mathcal{H}_{\nu-1}
 \end{pmatrix}\begin{pmatrix}\phi_+\\ \phi_-\end{pmatrix},\quad &\nu=\frac{k}{\sigma}+\frac12,\;k\geq1,\\\label{ipk-}
 \Pk^{-1}\phi(r)=2^{-\frac12} \begin{pmatrix}
 \mathcal{H}_{\nu}&\mathcal{H}_{\nu}\\
 -\mathcal{H}_{\nu+1}&\mathcal{H}_{\nu+1}
 \end{pmatrix}\begin{pmatrix}\phi_+\\ \phi_-\end{pmatrix},\quad &\nu=-(\frac{k}{\sigma}+\frac12),\;k\leq-1,
 \end{align}
and the fact $\mathcal{H}_{\nu}^{2}=I$ verifies
that this operator actually inverts $\mathcal{P}_{k}$.
It is also clear that $\mathcal{P}_{k}^{-1}=\mathcal{P}_{k}^{*}$
so that $\mathcal{P}_{k}$ is unitary.

By \eqref{eq:rH0}, for $k=0$ we restricted $\gamma$
such that $\sin \gamma\cos \gamma=0$.
Comparing with \eqref{pk+} and \eqref{pk-}, accordingly we have
if $\cos \gamma=0$ then $\mathcal{P}_{0}^{-1}$
is given by formula \eqref{ipk+} with $\nu=\frac12$,
while if $\sin \gamma=0$ then $\mathcal{P}_{0}^{-1}$
is given by formula \eqref{ipk-} with $\nu=-\frac12$ (we use the well known fact that the Hankel inversion formula holds as long as $\nu \geq-1/2$.\vspace{0.2cm}

In the end, we proved the relativistic Hankel transform.

\begin{proposition}\label{pro:hankel}
  Let $\sigma\in(0,1]$ and $\gamma\in[0,2\pi)$ such that
  \begin{equation}\label{bd:sincos}
    \sin \gamma\cos \gamma=0.
  \end{equation}
  Consider the selfadjoint operators $d^{\gamma}_{k}$
  defined in \eqref{eq:d0g}, \eqref{eq:dkg},
  with domains $Y$ for $k\neq0$ and $Y^{\gamma}$ for $k=0$.
  Define the relativistic Hankel transform
  $\mathcal{P}_{k}$ as in \eqref{eq:Pk},
  with $\Psi_{k}$ given by \eqref{Psik+}, \eqref{Psik-},
  \eqref{Psi0}.

  Then $\mathcal{P}_{k}$ can be written in the equivalent forms
  \eqref{pk+}, \eqref{pk-},
  \eqref{pk0}. Moreover,
  \begin{enumerate}
    \item
    $\mathcal{P}_{k}:
     [ L^{2}((0,+\infty),rdr]^{2}\to
      L^{2}(\mathbb{R},|\rho|d\rho)$
    is a bounded bijection.

    \item
    $\mathcal{P}_{k}$ is unitary.
    $\mathcal{P}_{k}^{-1}=\mathcal{P}^{*}_{k}$
    is given by \eqref{ipk+} for $k\ge1$ or
    $k=0$ and $\sin \gamma=0$, and by
    \eqref{ipk-} if $k\le-1$ or $k=0$ and
    $\cos \gamma=0$.

    \item
    We have
    $\mathcal{P}_{k}d_{k}^{\gamma}=\rho\mathcal{P}_{k}$
    on $Y^\gamma$ for all $k\in \mathbb{Z}$.
  \end{enumerate}
\end{proposition}


\subsection{The Construction of Dirac Propagator}\label{sec:dirapro}
To study the Cauchy problem associated to the Dirac Hamiltonian on the cosmic string, using the relativistic Hankel transform, we will show an explicit representation for the solution of the Dirac equation
 \begin{align}\label{eq:Dirac}
 \begin{cases}
 i\partial_t u=\D_{\sigma,\gamma} u,\qquad u(t,x):\R_t\times X\rightarrow \C^2,\\
 u(0,x)=f(x)=(\phi,\psi)^T\in D(\D_{\sigma,\gamma})\subset(L^2(X))^2.
 \end{cases}
 \end{align}
For any $f\in D(\D_{\sigma,\gamma})$, we recall the form
 \begin{align}
 f=\sum_{k\in\Z}\begin{pmatrix}\phi_k(r)e^{-i\frac{k}{\sigma}\theta}\\
 \psi_k(r)e^{-i\frac{k}{\sigma}\theta}\end{pmatrix}
 \end{align}
with
 \begin{align*}
 \begin{pmatrix}\phi_0(r)\\ \psi_0(r)\end{pmatrix}=c\begin{pmatrix}\cos\gamma\cdot K_{\frac12}(r)\\ \sin\gamma\cdot K_{\frac12}(r)\end{pmatrix}+\begin{pmatrix}\tilde{\phi}_0(r)\\ \tilde{\psi}_0(r)\end{pmatrix}
 \end{align*}
where $\begin{pmatrix}
  \tilde\phi_{0}(r) \\
    \tilde\psi_{0}(r)
\end{pmatrix}, \begin{pmatrix}
  \phi_{k}(r) \\
    \psi_{k}(r)
\end{pmatrix}\in Y$
for $k\neq0$ are regular at the origin.

We shall prove the following:

\begin{proposition}
[Dirac Propagator]\label{prop:Sp} Let $\gamma\in [0,2\pi)$ be such that $\sin\gamma \cos\gamma=0$. Let $x=r(\cos \theta,\sin\theta)$ and $y=s(\cos\omega,\sin\omega)$ and let

\begin{equation}\label{Phi}
\Phi_{k,\sigma}(\theta)=\frac1{\sqrt{2\pi\sigma}}\left(\begin{matrix} e^{-i\frac k\sigma \theta}&0\\0& e^{-i\frac k\sigma \theta}\end{matrix}\right).
\end{equation}
Then the solution $u(t,x)$ to
\eqref{eq:Dirac} satisfies
\begin{equation}\label{eq:ker}
u(t, x)=e^{-it\D_{\sigma,\gamma}}f=\int_{0}^\infty \int_{-\pi}^{\pi}{\bf K}_{\sigma}(t,r,\theta,s,\omega) f(s,\omega)  d\omega \;sds
\end{equation}
where
\begin{equation}\label{kernel-s}
\begin{split}
{\bf K}_{\sigma}(t,r,\theta,s,\omega)=\sum_{k\in\Z} \Phi_{k,\sigma}(\theta) {\bf K}_{k,\sigma}(t,r,s)\overline{\Phi_{k,\sigma}(\omega)}
\end{split}
\end{equation}
and for $k\neq0$
\begin{equation}\label{kernel'}
\begin{split}
{\bf K}_{k,\sigma}(t, r,s)= \int_0^\infty e^{-it\rho}  \Psi_k(r\rho)\overline{\Psi^*_k(s\rho)} \rho d\rho
\end{split}=\begin{pmatrix}m_{|\frac{k}{\sigma}+\frac12|}&0\\ 0&m_{|\frac{k}{\sigma}-\frac12|}\end{pmatrix}.
\end{equation}
In addition for $k=0$
 \begin{align}\label{kernel0}
 {\bf K}_{0}=\begin{pmatrix}m_{-\frac12}&0\\ 0&m_{\frac12}\end{pmatrix}\quad \text{if}\; \sin\gamma=0,\quad \begin{pmatrix}m_{\frac12}&0\\ 0&m_{-\frac12}\end{pmatrix}\quad \text{if}\; \cos\gamma=0,
 \end{align}
where
 \begin{align}
 m_{\nu}(t,r,s)=\int^\infty_0 e^{-it\rho}J_{\nu}(r\rho)J_{\nu}(s\rho)\rho d\rho.
 \end{align}
\end{proposition}

\begin{proof}
From the decomposition \eqref{L2Domc} we expand the initial data as
$$f(x)=(\phi,\psi)^{T}=\sum_{k\in\Z}\Phi_k(\theta) f_k(r),$$
where $f_k(r)=(\phi_k(r), \psi_k(r))^T$, $\Phi_k(\theta)$ is in \eqref{Phi} and
\begin{equation}\label{eq:fk}
 \phi_k(r)=\int_{-\pi}^{\pi} \phi(x) e^{-i\frac{k}{\sigma}\theta}\, d\theta, \quad \psi_k(r)=\int_{-\pi}^{\pi} \psi(x) e^{-i\frac{k}{\sigma}\theta}\, d\theta.
\end{equation}
hence we get the solution of \eqref{eq:Dirac}
\begin{equation}
u(t, x)=e^{-it\mathcal{D}_{A,\gamma}}f=\sum_{k\in\Z}\Phi_k(\theta) u_k(t,r),\qquad \forall k\in\Z,
\end{equation}
and $u_k(t, r)$ satisfies the ordinary differential equations
\begin{equation}\label{eq:uka}
\begin{cases}
i\partial_t u_k(t, r)=d^{\gamma}_{k} u_k(t,r), \\
u_k|_{t=0}=f_k(r)=(\phi_{k}(r),\psi_{k}(r))^T \in  D(d^{\gamma}_{k}).
\end{cases}
\end{equation}
We recall that  the operator $d^{\gamma}_{k}$ is given by
$$
d^{\gamma}_{k}=
\begin{cases}
d^{\gamma}_{0} \quad {\rm if}\: k=0,\\
\overline{d}_{k}\quad {\rm if}\: k\neq 0
\end{cases}
$$
with $\gamma$ as in the statement.
 Denote
\begin{equation*}
  v_{k}(t,\rho)=\mathcal{P}_{k}u_{k},
\end{equation*}
then, by the diagonalization property in Proposition \ref{pro:hankel} the function $v_{k}$ solves the following equations
\begin{equation*}
  \begin{cases}
    i \partial_{t}v_{k}(t,\rho)=\rho v_{k}(t,\rho)\\
    v_{k}\vert_{t=0}=\mathcal{P}_{k}f_{k}(\rho)
  \end{cases}
\end{equation*}
Thus this gives immediately
\begin{equation*}
  v_{k}(t,\rho)=e^{-it\rho}\mathcal{P}_{k}f_{k}(\rho)
\end{equation*}
and
\begin{equation*}
  u_{k}(t,x)=\mathcal{P}_{k}^{-1}
  e^{-it\rho}\mathcal{P}_{k}f_{k}(\rho).
\end{equation*}
Thus the Dirac propagator $e^{-it\overline{d}_k}$ can be equivalently characterized as kernel $\mathbf{K}_{k,\sigma}(t,r,s)$
of the operator
$\mathcal{P}_{k}^{-1}e^{-it\rho}\mathcal{P}_{k}$.
Recalling the expressions \eqref{pk+} and
\eqref{ipk+} for  $k\ge1$, or \eqref{pk0} for $k=0$,
$\cos\gamma=0$ we obtain
\begin{equation*}
  \mathcal{P}_{k}^{-1}e^{-it \rho}\mathcal{P}_{k}=
  \begin{pmatrix}
    \mathcal{H}_{\nu}e^{-it |\rho|}\mathcal{H}_{\nu} & 0 \\
    0 &  \mathcal{H}_{\nu-1}e^{-it |\rho|}\mathcal{H}_{\nu-1}
  \end{pmatrix},
  \qquad
  \nu=
  \begin{cases}
    \frac{k}{\sigma}+\frac12 &\text{if $ k\ge1 $,}\\
    \frac12 &\text{if $ k=0 $.}
  \end{cases}
\end{equation*}
On the other hand, for $k\le-1$ with \eqref{pk-}, \eqref{ipk-}
or $k=0$, $\sin\gamma=0$ with \eqref{pk0}, we obtain
\begin{equation*}
  \mathcal{P}_{k}^{-1}e^{-it \rho}\mathcal{P}_{k}=
  \begin{pmatrix}
    \mathcal{H}_{\nu}e^{-it |\rho|}\mathcal{H}_{\nu} & 0 \\
    0 &  \mathcal{H}_{\nu+1}e^{-it |\rho|}\mathcal{H}_{\nu+1}
  \end{pmatrix},
  \qquad
  \nu=
  \begin{cases}
    -(\frac{k}{\sigma}+\frac12) &\text{if $ k\le-1 $,}\\
    -\frac12 &\text{if $ k=0 $.}
  \end{cases}
\end{equation*}
Define the scalar operator
$\mathcal{M}_{\nu}$:
\begin{equation*}
  \mathcal{M}_{\nu}\phi(r)=
  \mathcal{H}_{\nu}e^{-it |\rho|}\mathcal{H}_{\nu}\phi(r),
\end{equation*}
hence $\mathcal{M}_\nu$ has the integral form
\begin{equation*}
  \mathcal{M}_{\nu}\phi(r)=
  \int_{0}^{\infty}
  m_{\nu}(t,r,s)\phi(s)sds
\end{equation*}
with the kernel
\begin{equation*}
  m_{\nu}(t,r,s)=
  \int_{0}^{\infty}e^{-it \rho}
  J_{\nu}(r \rho)J_{\nu}(s\rho)\rho d \rho.
\end{equation*}
Therefore we rewrite for $k\neq0$
\begin{equation*}
  \mathcal{P}_{k}^{-1}e^{-it \rho}\mathcal{P}_{k}=
  \begin{pmatrix}
    \mathcal{M}_{|\frac{k}{\sigma}+\frac12|} & 0 \\
    0 &  \mathcal{M}_{|\frac{k}{\sigma}-\frac12|}
  \end{pmatrix}
\end{equation*}
but for $k=0$ we have
\begin{equation*}
  \mathcal{P}_{0}^{-1}e^{-it \rho}\mathcal{P}_{0}=\begin{cases}
  \begin{pmatrix}
    \mathcal{M}_{\frac12} & 0 \\
    0 &  \mathcal{M}_{-\frac12}
  \end{pmatrix}
  \quad\text{if $\cos\gamma=0$,}\\

  \begin{pmatrix}
    \mathcal{M}_{-\frac12} & 0 \\
    0 &  \mathcal{M}_{\frac12}
  \end{pmatrix}
  \quad\text{if $\sin\gamma=0$.}
  \end{cases}
\end{equation*}
Therefore, we obtain \eqref{kernel'} and \eqref{kernel0} and we complete the proof.

\end{proof}
Since the initial condition $f\in D(\D_{d^\gamma_k})$ has the regular part and singular part in the meantime, it is difficult to directly treat the propagator in a standard way (more details see \cite{FZZ,GYZZ}).
At this point, we shall deal with the initial conditions for $k\neq0$ (nonradial part) and $k=0$ (radial part) separately to prove our results.
Using the orthogonal projector $P_{\bot}$ defined at the beginning
of Section \ref{sec:results}, we can spilt $f$ into nonradial and radial two parts: $f=f_{\bot}+f_0$, where
\begin{equation*}
  f_{\bot}=P_{\bot}f=  \sum_{k\in \mathbb{Z}\setminus\{0\}}
  \begin{pmatrix}
    e^{-i\frac{k}{\sigma} \theta}\ \phi_{k}(r) \\
    e^{-i\frac{k}{\sigma}\theta}\ \psi_{k}(r)
  \end{pmatrix}.
\end{equation*}
While the {\em radial part} is more
delicate due to the singularity at 0.
This corresponds to the choice of data
\begin{equation*}
  f_0=P_{0}f=
  \begin{pmatrix}\phi_0(r)\\ \psi_0(r)\end{pmatrix}=c\begin{pmatrix}\cos\gamma\cdot K_{\frac12}(r)\\ \sin\gamma\cdot K_{\frac12}(r)\end{pmatrix}+\begin{pmatrix}\tilde{\phi}_0(r)\\ \tilde{\psi}_0(r)\end{pmatrix}.
\end{equation*}
By the orthogonality, the full flow is the sum of the nonradial and the radial components:
 \begin{equation}\label{eq:flow}
e^{-it\mathcal{D}_{\sigma,\gamma}} f=e^{-it\mathcal{D}_{\sigma,\gamma}}P_0 f+e^{-it\mathcal{D}_{\sigma,\gamma}}P_{\bot} f=e^{-it\mathcal{D}_{\sigma,\gamma}} f_0+e^{-it\mathcal{D}_{\sigma,\gamma}} f_{\bot}.
  \end{equation}

\section{Localized dispersive estimates }
\label{sec:dispest}
In this section, we prove Theorem \ref{th:disp}.
Recall \eqref{eq:square}
\begin{align*}
H=\begin{pmatrix}H^+_\sigma&0\\ 0&H^-_\sigma\end{pmatrix}=\bigoplus_{k\in\Z}L_k\otimes I_2,
 \end{align*}
where
 \begin{align*}
L_k=\begin{pmatrix} H_{\nu_+}&0\\0&H_{\nu_-}\end{pmatrix}.
 \end{align*}
 For the $H_{\nu_\pm}=-\frac{d^2}{dr^2}-\frac1r\frac{d}{dr} +\frac1{r^2}(\frac12\pm\frac{k}{\sigma})^2$ and the Hankel transform $\mathcal{H}_{\nu_\pm}$ given in \eqref{Hankel}, we have
 \begin{align*}
 H_{\nu_\pm}=\mathcal{H}_{\nu_\pm}\rho^2\mathcal{H}_{\nu_\pm}.
 \end{align*}
By Proposition \ref{pro:hankel} on the relativistic Hankel transform, we obtain
 \begin{align}
L_k=\mathcal{P}^{-1}_k\rho^2\mathcal{P}_k, \quad \text{for all} \; k\in\Z, k\neq0,
 \end{align}
 while for $k=0$
 \begin{align}
L_0=\begin{pmatrix} \mathcal{H}_{\frac12}\rho^2\mathcal{H}_{\frac12}&0\\0&\mathcal{H}_{\frac12}\rho^2\mathcal{H}_{\frac12}\end{pmatrix}:=\tilde{\mathcal{P}}^{-1}_0\rho^2\tilde{\mathcal{P}}_0
 \end{align}
 which is different from $\mathcal{P}^{-1}_0\rho^2\mathcal{P}_0$. This difference is because we have no the singular component to treat due the Friedrichs extension of $H$, see Remark \ref{rem:H}.  Hence we obtain
  \begin{align}\label{e-Lk}
 e^{-it\sqrt{L_k}}=\mathcal{P}_k^{-1}e^{-it|\rho|}\mathcal{P}_k,\quad e^{-it\sqrt{L_0}}=\tilde{\mathcal{P}}^{-1}_0e^{-it|\rho|}\tilde{\mathcal{P}}_0.
  \end{align}
Therefore, we have the decomposition of the propagator $e^{-it\sqrt{H}}$
 \begin{align}
 e^{-it\sqrt{H}}=\bigoplus_{k\in\Z} e^{-it\sqrt{L_k}} \otimes I_2
 \end{align}
where $e^{-it\sqrt{L_k}}$ given in \eqref{e-Lk}. It worths to pointing out that $ e^{-it\sqrt{H}}$ equals $e^{-it\D_{\sigma,\gamma}}$ except on the $h_{0}(\mathbb{S}_\sigma^{1})$ given in \eqref{L2Domc}, which will allow us to prove
\eqref{est:Pk}.

Let $P_>$ and $P_<$ be the projections on the harmonics with $k\geq1$ and $k\leq-1$ respectively. In view of this decomposition we can rewrite \eqref{eq:flow} as
 \begin{align}\label{m:proH}
 e^{-it\D_{\sigma,\gamma}}f=e^{-it\sqrt{H}}P_> f+e^{-it\sqrt{H}}P_< f+e^{-it\D_{\sigma,\gamma}}P_0f.
 \end{align}
Therefore, we also have for all $\varphi\in C^\infty_0(\R)$
 \begin{align}\label{m:proH'}
 \varphi(2^{-j}|\D_{\sigma,\gamma}|)e^{-it\D_{\sigma,\gamma}}P_{\bot}f= \varphi(2^{-j}\sqrt{H})e^{-it\sqrt{H}}P_{\bot}f,\quad j\in\Z.
 \end{align}
In the following discussion, we split the propagator into two parts to prove Theorem \ref{th:disp}.
\subsection{The proof of dispersive estimate \eqref{est:Pk}}
In this subsection, we prove \eqref{est:Pk}. In view of \eqref{m:proH'}, to prove \eqref{est:Pk},
we  study the half-wave operator associated with the operator $H$
$$\varphi(2^{-j}\sqrt{H})e^{-it\sqrt{H}}.$$
From the diagonally formula \eqref{eq:square}, it suffices to prove
\begin{equation}\label{est:Pk'}
\begin{split}
 \Big\|\varphi(2^{-j}\sqrt{H_{\sigma}^\pm})&e^{-it\sqrt{H_{\sigma}^\pm}}P_{\bot}f(x)\Big\|_{L^\infty(X)}
 \\& \leq C2^{2j}(1+2^{j}|t|)^{-\frac12}\|\tilde{\varphi}(2^{-j}\sqrt{H_{\sigma}^\pm})P_{\bot}f\|_{L^1(X)},
 \end{split}
\end{equation}
where the spinor operator $H$ is replaced by the scalar one  $H^\pm_\sigma=-\partial_r^2-\frac 1r\partial_r+\frac1{r^2}(\frac12\pm i\partial_\theta)^2$ on the flat cone
$X=(0,+\infty)\times\mathbb{S}^1_{\sigma}$.
Illustrated by \cite{BFM, Ford, Zhang1}, it will be easier to construct the Schr\"odinger kernel than the half-wave propagator associated with the operator $H_\sigma^\pm$.

Fortunately, the half-wave propagator can be linked with the Schr\"odinger kernel using the following subordination formula.
The following proposition about the subordination formula are modified from \cite[Proposition 4.1]{MS} and \cite[Proposition 2.2]{DPR}, and we use the one formulated in \cite{WZZ2}.

 \begin{proposition}\label{prop:sub}
 Let $\varphi\in C^\infty_c$ is supported in $[\frac12,2]$. There exist a symbol $\chi\in C^\infty(\R\times\R)$ with  supp$\chi(\cdot,\tau)\subseteq[\frac1{16},4]$ satisfying
  \begin{align}
  \sup_{\tau\in\R}\big|\partial_s^\alpha\partial_\tau^\beta \chi(s,\tau)\big|\lesssim_{\alpha,\beta}(1+|s|)^{-\alpha},\quad \forall \alpha,\beta\geq0.
  \end{align}
 and $\beta(s,\tau)$ is a Schwartz class function such that
  \begin{equation}
  \begin{aligned}
  \varphi(2^{-j}\sqrt{x})e^{it\sqrt{x}}=&\varphi(2^{-j}\sqrt{x})\beta( \frac{tx}{2^j},2^jt)\\
  &+\varphi(2^{-j}\sqrt{x})(2^jt)^{\frac12} \int^\infty_0\chi(s,2^jt)e^{\frac{i2^jt}{4s}}e^{i2^{-j}tsx}ds.
  \end{aligned}
  \end{equation}
 \end{proposition}

 With this in hand, then by the spectral theory for the non-negative self-adjoint operator $H_{\sigma}^\pm$, we can have the representation of the microlocalized half-wave propagator
\begin{equation}\label{key-operator}
\begin{split}
&\varphi(2^{-j}\sqrt{H_{\sigma}^\pm})e^{it\sqrt{H_{\sigma}^\pm}}\\&=\rho\big(\frac{tH_{\sigma}^\pm}{2^j}, 2^jt\big)
+\varphi(2^{-j}\sqrt{H_{\sigma}^\pm})\big(2^jt\big)^{\frac12}\int_0^\infty \chi(s,2^jt)e^{\frac{i2^jt}{4s}}e^{i2^{-j}tsH_{\sigma}^\pm}\,ds.
\end{split}
\end{equation}

To prove \eqref{est:Pk'}, we need three propositions. This first one is dispersive estimates for the Schr\"odinger propagator and the other two are about the Littlewood-Paley theory associated with the operator $H_{\sigma}^\pm$.
\begin{proposition}[Dispersive estimate]\label{prop:DS}
The Schr\"odinger propagator $e^{-itH_{\sigma}^\pm}$ satisfies
 \begin{align}\label{dis:KS}
 \|e^{-itH_{\sigma}^\pm}f\|_{L^\infty(X)}\lesssim |t|^{-1}\|f\|_{L^1(X)}.
 \end{align}
\end{proposition}

\begin{proposition}\label{prop:LP0}
Let $\varphi\in C_{c}^{\infty}(\mathbb{R})$ be supported in $[\frac12,2]$. Then for all $f\in L^q(X)$ and $j\in\Z$ there have
 \begin{align}\label{est:bern}
 \|\varphi(2^{-j}\sqrt{H^\pm_\sigma})f\|_{L^p(X)}\lesssim 2^{2j(\frac1q-\frac1p)}\|f\|_{L^q(X)},
 \end{align}
 provided $1\leq q< p\leq\infty$ or $1< q\leq p<\infty$.
\end{proposition}

\begin{proposition}\label{prop:LP}
  Let $\sum_{j\in \mathbb{Z}}\varphi(2^{-j}x)=1$ be a standard dyadic
  decomposition of unity. Then for any $1<p<\infty$, we have
 \begin{align}\label{est:LP}
 c_p\|f\|_{L^p(X)}\leq\Big\|\Big(\sum_{j\in\Z} |\varphi(2^{-j}\sqrt{H^\pm_\sigma})f|^2\Big)^{\frac12}\Big\|_{L^p(X)}\leq C_p\|f\|_{L^p(X)}
 \end{align}
where $c_p$ and $C_p$ are constants depending on $p$.
\end{proposition}

We now prove \eqref{est:Pk'}. To light the notation, without confusion, we briefly write $H$ as $H_{\sigma}^\pm$ in the rest of this subsection.
We estimate  the microlocalized half-wave propagator $$ \big\|\varphi(2^{-j}\sqrt{H})e^{it\sqrt{H}}P_{\bot} f\big\|_{L^\infty(X)}$$
by considering two cases that: $|t|2^j\geq 1$ and $|t|2^{j}\leq 1$. In the following argument, as before, we can choose $\tilde{\varphi}\in C_c^\infty((0,+\infty))$ such that $\tilde{\varphi}(\lambda)=1$ if $\lambda\in\mathrm{supp}\,\varphi$
and $\tilde{\varphi}\varphi=\varphi$. Since $\tilde{\varphi}$ has the same property of $\varphi$, without confusion, we drop off the tilde above $\varphi$ for brief. Without loss of generality, in the following argument, we assume $t>0$.
\vspace{0.1cm}

\underline{{\bf Case 1: $2^jt\leq 1$.}}
From the Bernstein inequality \eqref{est:bern} and the fact $e^{-it\sqrt{H}}$ is bounded on $L^2$, we have
 \begin{equation}
 \begin{aligned}
 \|&\varphi(2^{-j}\sqrt{H})e^{it\sqrt{H}}P_{\bot} f\|_{L^\infty(X)}\\
 &\lesssim 2^{j}\|e^{-it\sqrt{H}}\varphi(2^{-j}\sqrt{H})P_{\bot}f\|_{L^\infty(X)}\\
 &\lesssim 2^{j}\|\varphi(2^{-j}\sqrt{H})P_{\bot}f\|_{L^\infty(X)}
 \lesssim 2^{j}\|P_{\bot}f\|_{L^\infty(X)}.
 \end{aligned}
 \end{equation}
Since $0<2^j t\leq1$, one has
 \begin{equation}\label{est:hw1}
 \begin{aligned}
 \|&\varphi(2^{-j}\sqrt{H})e^{it\sqrt{H}}P_{\bot} f\|_{L^\infty(X)}\lesssim 2^{2j}(1+2^jt)^{-N}\|\varphi(2^{-j}P_{\bot}f\|_{L^1(X)}.
 \end{aligned}
 \end{equation}
\underline{{\bf Case 2: $2^jt\geq 1$.}}
From \eqref{key-operator}, by using the fact that $\beta\in\mathcal{S}(\R\times\R)$ and the Bernstein inequality again, we have
 \begin{equation}
 \begin{aligned}
 &\|\varphi(2^{-j}\sqrt{H})\beta(\frac{tH}{2^j},2^jt)P_{\bot}f\|_{L^\infty(X)}\\
 &\lesssim2^j(1+2^jt)^{-N}\|\varphi(2^{-j}\sqrt{H}) P_{\bot}f\|_{L^2(X)}\\&
 \lesssim2^{2j}(1+2^jt)^{-N}\|\varphi(2^{-j}\sqrt{H}) P_{\bot}f\|_{L^1(X)}.
 \end{aligned}
 \end{equation}
 For the other part, by \eqref{dis:KS}, we obtain
 \begin{equation}\label{est:hw2}
 \begin{aligned}
 \|&\varphi(2^{-j}\sqrt{H})(2^jt)^{\frac12} \int^\infty_0\chi(s,2^jt)e^{\frac{i2^jt}{4s}}e^{i2^{-j}tsH}ds P_{\bot}f(x) \|_{L^\infty(X)}\\
 &\lesssim (2^jt)^{\frac12}\int^\infty_0\chi(s,2^jt)|2^{-j}ts|ds \|\varphi(2^{-j}\sqrt{H})P_{\bot}f\|_{L^1(X)}\\
 &\lesssim2^{2j}(2^{j}t)^{-1} \|\varphi(2^{-j}\sqrt{H})P_{\bot}f\|_{L^1(X)}\\
 &\lesssim2^{2j} (1+2^{-j}t)^{-\frac12} \|\varphi(2^{-j}\sqrt{H})P_{\bot}f\|_{L^1(X)}.
 \end{aligned}
 \end{equation}
Then, we complete the proof of \eqref{est:Pk'}, hence  \eqref{est:Pk} has been proved, if we could prove Proposition \ref{prop:DS},  Proposition \ref{prop:LP0}, and   Proposition \ref{prop:LP}.

\begin{proof}[The proof of Proposition \ref{prop:DS}] The proof is modified from \cite{Zhang1} but we should be careful of the role of the parameters $\sigma$. Note that
$$H^\pm_\sigma=-\partial_r^2-\frac 1r\partial_r+\frac1{r^2}(\frac12\pm i\partial_\theta)^2,$$
 is a bit different from the operator $-\Delta_g$ considered in \cite{Zhang1}. The difference leads to a shift of the eigenvalues. Using the spectrum calculus of Cheeger-Taylor \cite{CT1,CT2} (also see \cite{Zhang1}),
and using the coordinate
\begin{equation}\label{eq:coordinate}
x=(r, \theta),\quad y=(s,\omega),
\end{equation}
we obtain the Schr\"odinger kernel $e^{-it H_{\sigma}^\pm}(x,y)$
  \begin{align}\label{equ:ktxyschr}
   K_{\sigma}^\pm(t, r,\theta,s,\omega)
    =&\frac1{2\pi\sigma} \sum_{k\in\Z} e^{-i\frac k\sigma(\theta-\omega)}  K_{\sigma}^\pm(k; t, r, s),
  \end{align}
  where
  \begin{equation}\label{KS}
\begin{split}
K_{\sigma}^\pm(k; t, r, s)=\int_0^\infty e^{-it\rho^2}J_{\nu_{\pm}(k)}(r \rho)J_{\nu_{\pm}(k)}(s\rho) \,\rho d\rho,\quad \nu_\pm(k)=\big|\frac{k}{\sigma}\pm\frac12\big|.
\end{split}
\end{equation}
Writing $\nu_\pm= \nu_\pm(k)$ and using the Weber identity
(see e.g. \cite[Proposition 8.7]{Taylor}) and  analytic continuation  like in \cite[Page 161, (8.88)]{Taylor}, we have
 \begin{align}\label{eq:K-bessel}
K_{\sigma}^\pm(k; t, r, s)=&\int_0^\infty e^{-it\rho^2}J_{\nu_{\pm}}(r \rho)J_{\nu_{\pm}}(s\rho) \,\rho d\rho\\\nonumber
  =&\lim_{\epsilon\searrow0}\int_0^\infty e^{-(\epsilon+it)\rho^2}J_{\nu_{\pm}}(r \rho)J_{\nu_{\pm}}(s\rho) \,\rho d\rho\\\nonumber
  =&\lim_{\epsilon\searrow0}\frac{e^{-\frac{r^2+s^2}{4 (\epsilon+it)}}}{2(\epsilon+it)} I_{\nu_{\pm}}\Big(\frac{rs}{2(\epsilon+it)}\Big)=\frac{e^{-\frac{r^2+s^2}{4 it}}}{2it} I_{\nu_{\pm}}\Big(\frac{rs}{2it}\Big),
\end{align}
where $I_\nu$ is the modified Bessel function.
For $z=\frac{rs}{2(\epsilon+it)}$ with $\epsilon>0$, we use the integral representation of the modified Bessel function $I_\nu$ in \cite{Taylor}  to write
\begin{equation}\label{m-bessel}
I_{\nu_{\pm}}(z)=\frac1{\pi}\int_0^\pi e^{z\cos \tau} \cos(\nu_{\pm} \tau) d\tau-\frac{\sin(\nu_{\pm}\pi)}{\pi}\int_0^\infty e^{-z\cosh \tau} e^{-\tau\nu_{\pm}} d\tau.
\end{equation}
To prove \eqref{dis:KS}, it suffices to estimate
 \begin{equation}\label{eq:I1}
 \frac1{2\pi\sigma} \frac1\pi \sum_{k\in\Z} e^{-i\frac k\sigma (\theta-\omega)} \int_0^{\pi} e^{\frac{rs}{2it}\cos \tau} \cos(\nu_\pm(k) \tau) d\tau,
  \end{equation}
  and
    \begin{equation}\label{eq:I2}
 \frac1{2\pi\sigma} \frac1\pi \sum_{k\in\Z} e^{-i\frac k\sigma (\theta-\omega)} \sin(\nu_{\pm}(k) \pi)\int_0^\infty e^{-\frac{rs}{2it}\cosh \tau} e^{-\tau\nu_\pm(k)} d\tau.
  \end{equation}
 We first estimate \eqref{eq:I1}. Recall $\nu_\pm(k)=|\frac{k}{\sigma}\pm\frac12|$, we compute the summation
 \begin{align*}
 &\sum_{k\in\Z} \frac1{2\pi\sigma}\cos(\nu_\pm \tau)e^{-i\frac k\sigma (\theta-\omega)}=\sum_{k\in\Z} \frac1{4\pi\sigma}e^{-i\frac k\sigma (\theta-\omega)}\big(e^{i\nu_\pm\tau}+e^{-i\nu_\pm\tau}\big)\\
 &=\frac12\sum_{j\in\Z}\left(e^{\pm i\frac12 \tau}\delta(\theta-\omega-\tau+2j \pi\sigma)+e^{\mp i\frac12 \tau}\delta(\theta-\omega+\tau+2j \pi\sigma)  \right)
 \end{align*}
 where we use the fact $0<\sigma\leq 1$ and the Poisson summation formula
 \begin{align}
 \sum_{j\in\Z}\delta(t-jT)=\sum_{k\in\Z}\frac1T e^{i2\pi\frac{k}T t},\quad T=2\pi\sigma.
 \end{align}
Putting this into \eqref{eq:I1}, we estimate
 \begin{equation}\label{est:I1}
 \begin{split}
\eqref{eq:I1}&\lesssim  \sum_{j\in\Z}\Big|\int^{\pi}_0 e^{\frac{rs}{2it}\cos \tau}[e^{\pm i\frac12 \tau}\delta(\theta-\omega-\tau+2 j\sigma \pi)+e^{\mp i\frac12 \tau}\delta(\theta-\omega+\tau+2j\sigma \pi)]d\tau\Big|\\
 &\lesssim \sum_{\{j\in\Z:0\leq|\theta-\omega+2j\pi\sigma|\leq\pi\}} 1\lesssim_\sigma 1.
 \end{split}
 \end{equation}
 In the last inequality, we use the fact that the summation is a finite summation due to $\theta,\omega\in [-\pi\sigma,\sigma\pi)$.
Next we consider the second term \eqref{eq:I2}
 \begin{align}\nonumber
 &\sum_{k\in\Z} e^{-i\frac k\sigma(\theta-\omega)}\frac{\sin(\nu_\pm \pi)}{\pi}\int_0^\infty e^{-\frac{rs}{2it}\cosh \tau} e^{-\tau\nu_\pm} d\tau.
 \end{align}
Recall $ k\in\Z,~ k\neq0$, we obtain
\begin{equation}
\begin{split}
\sin(\pi\nu_\pm)=\sin(|\frac k\sigma\pm \frac12|\pi)=\begin{cases}
\pm\cos(\frac k\sigma\pi),\qquad &  k\geq1;\\
\mp\cos(\frac k\sigma\pi),\qquad &  k\leq -1.
\end{cases}
\end{split}
\end{equation}
Therefore we have
 \begin{align*}
 &\sum_{k\in\Z}\sin(|\frac k\sigma\pm\frac12|\pi)e^{-|\frac k\sigma\pm\frac12|\tau} e^{-i\frac k\sigma(\theta-\omega)}\\
 &=e^{-\frac\tau 2}+\sum_{k\geq1}\pm\frac{e^{i\frac k\sigma\pi}+e^{-i\frac k\sigma\pi}}{2}e^{-(\frac k\sigma\pm\frac12)\tau}e^{-i\frac k\sigma(\theta-\omega)}+\sum_{k\leq-1}\mp\frac{e^{i\frac k\sigma\pi}+e^{-i\frac k\sigma\pi}}{2}e^{(\frac k\sigma\pm\frac12)\tau}e^{-i\frac k\sigma(\theta-\omega)}\\
 &=e^{-\frac\tau 2}\sum_{k\geq1}\pm\frac{e^{i\frac k\sigma\pi}+e^{-i\frac k\sigma\pi}}{2}e^{-(\frac k\sigma\pm\frac12)\tau}e^{-i\frac k\sigma(\theta-\omega)}+\sum_{k\geq1}\mp\frac{e^{i\frac k\sigma\pi}+e^{-i\frac k\sigma\pi}}{2}e^{-\frac k\sigma\tau\pm\frac12\tau}e^{i\frac k\sigma(\theta-\omega)}\\
 &=e^{-\frac\tau 2}\pm\frac{e^{\mp\frac12\tau}}2\sum_{k\geq1}e^{-\frac{k}\sigma\tau}\big(e^{-i\frac k\sigma(\theta-\omega+\pi)}+e^{-i\frac k\sigma(\theta-\omega-\pi)}\big)\\
 &\qquad\qquad\qquad\qquad\qquad\mp\frac{e^{\pm\frac12\tau}}2\sum_{k\geq1} e^{-\frac{k}\sigma\tau}\big(e^{i\frac k\sigma(\theta-\omega+\pi)}+e^{i\frac k\sigma(\theta-\omega-\pi)}\big).
 \end{align*}
Note that
 \begin{align}
 \sum_{k=1}^\infty e^{ikz}=\frac{e^{iz}}{1-e^{iz}},\qquad \mathrm{Im} z>0,
 \end{align}
we finally obtain
 \begin{equation}\label{est:I2}
 \begin{aligned}
 &\sum_{k\in\Z,k\neq0}\sin(|\frac k\sigma\pm\frac12|\pi)e^{-|\frac k\sigma\pm\frac12|\tau} e^{-i\frac k\sigma(\theta-\omega)}\\
 &=\pm\frac{e^{\mp\frac12\tau}}2\left( \frac{e^{-\frac\tau\sigma-i\frac{\theta-\omega+\pi}\sigma}}{1- e^{-\frac\tau\sigma-i\frac{\theta-\omega+\pi}\sigma}}+ \frac{e^{-\frac\tau\sigma-i\frac{\theta-\omega-\pi}\sigma}}{1- e^{-\frac\tau\sigma-i\frac{\theta-\omega-\pi}\sigma}} \right)\\ &\qquad\qquad\mp\frac{e^{\pm\frac12\tau}}2\left( \frac{e^{-\frac\tau\sigma+i\frac{\theta-\omega+\pi}\sigma}}{1- e^{-\frac\tau\sigma+i\frac{\theta-\omega+\pi}\sigma}}+ \frac{e^{-\frac\tau\sigma+i\frac{\theta-\omega-\pi}\sigma}}{1- e^{-\frac\tau\sigma+i\frac{\theta-\omega-\pi}\sigma}} \right)\\
 &=\frac14\Big(\frac{ e^{-\frac\tau\sigma}2\sinh\frac\tau2 \pm e^{\mp\frac\tau2-i\frac{\theta-\omega+\pi}\sigma} \mp e^{\pm\frac\tau2+i\frac{\theta-\omega+\pi}\sigma} } {\cosh\frac\tau\sigma- \cos\frac{\theta-\omega+\pi}\sigma} \\
 &\qquad\qquad+ \frac{ e^{-\frac\tau\sigma}2\sinh\frac\tau2 \pm e^{\mp\frac\tau2-i\frac{\theta-\omega-\pi}\sigma} \mp e^{\pm\frac\tau2+i\frac{\theta-\omega-\pi}\sigma}} {\cosh\frac\tau\sigma- \cos\frac{\theta-\omega-\pi}\sigma} \Big)\\
 &=:B_\pm(\tau,\theta,\omega,\sigma).
 \end{aligned}
 \end{equation}
 Plug this into \eqref{eq:I2}, we aim to prove
  \begin{align}\label{est:B}
\int_0^\infty \big(\big|B_\pm(\tau,\theta,\omega,\sigma)\big|+e^{-\frac\tau 2}\big) d\tau\lesssim 1.
 \end{align}
 If this can be done, we collect \eqref{eq:K-bessel} and \eqref{est:I1} to obtain \eqref{dis:KS}. Hence we are left to prove \eqref{est:B}. Obviously, one has $\int_0^\infty e^{-\frac\tau 2} d\tau\lesssim 1$, so it suffices to prove
  \begin{align}
 \int^\infty_0\Big|\frac{e^{-\frac\tau\sigma}2\sinh\frac\tau2 \pm e^{\mp\frac\tau2-i\frac{\theta-\omega+\pi}\sigma} \mp e^{\pm\frac\tau2+i\frac{\theta-\omega+\pi}\sigma}} {\cosh\frac\tau\sigma- \cos\frac{\theta-\omega+\pi}\sigma}\Big|d\tau\lesssim 1,
 \label{est:B1}\\
  \int^\infty_0\Big|\frac{e^{-\frac\tau\sigma}2\sinh\frac\tau2 \pm e^{\mp\frac\tau2-i\frac{\theta-\omega-\pi}\sigma} \mp e^{\pm\frac\tau2+i\frac{\theta-\omega-\pi}\sigma}} {\cosh\frac\tau\sigma- \cos\frac{\theta-\omega-\pi}\sigma}\Big|d\tau\lesssim 1,
 \end{align}
 where the implicit constant is independent of $\theta, \omega$. We only prove the first one since the second one can be proved by the same argument. We observe that
  \begin{align*}
 \cosh\frac\tau\sigma- \cos\frac{\theta-\omega+\pi}\sigma=2\sinh^2\Big(\frac\tau{2\sigma}\Big) +2\sin^2\Big(\frac{\theta-\omega+\pi}{2\sigma}\Big).
 \end{align*}
 To prove \eqref{est:B1},  we aim to show
  \begin{align}
\int^1_0 \Big|\frac{e^{-\frac\tau\sigma}2\sinh\frac\tau2 \pm e^{\mp\frac\tau2-i\frac{\theta-\omega+\pi}\sigma} \mp e^{\pm\frac\tau2+i\frac{\theta-\omega+\pi}\sigma}}  {\sinh^2(\frac\tau{2\sigma}) +\sin^2(\frac{\theta-\omega+\pi}{2\sigma})}\Big|d\tau \lesssim 1,\label{est:B11}\\
\int^\infty_1 \Big|\frac{e^{-\frac\tau\sigma}2\sinh\frac\tau2 \pm e^{\mp\frac\tau2-i\frac{\theta-\omega+\pi}\sigma} \mp e^{\pm\frac\tau2+i\frac{\theta-\omega+\pi}\sigma}}  {\sinh^2(\frac\tau{2\sigma}) +\sin^2(\frac{\theta-\omega+\pi}{2\sigma})}\Big|d\tau \lesssim 1. \label{est:B12}
 \end{align}
 The left hand side of \eqref{est:B12} can be estimated by
  \begin{align*}
\lesssim \int^\infty_1 e^{-(\frac2\sigma-\frac12)\tau}+e^{-(\frac1\sigma\mp\frac12)\tau} +e^{-(\frac1\sigma\pm\frac12)\tau}d\tau\lesssim 1,
 \end{align*}
 where we use the fact that $0<\sigma\leq 1$. Hence we have proved \eqref{est:B12}. The proof of \eqref{est:B11} is a bit complicate due to the singularity of the denominator, as $\tau\to 0$ and $\sin^2(\frac{\theta-\omega+\pi}{2\sigma})\to 0$.
 Fortunately, we observe that
  \begin{align*}
  &e^{-\frac\tau\sigma}2\sinh\frac\tau2 \pm e^{\mp\frac\tau2-i\frac{\theta-\omega+\pi}\sigma} \mp e^{\pm\frac\tau2+i\frac{\theta-\omega+\pi}\sigma}\\
  &=2 e^{-\frac\tau\sigma}\sinh\frac\tau2- 2 \sinh\frac\tau2\cos\big(\frac{\theta-\omega+\pi}\sigma\big)\mp i 2\cosh\frac\tau2 \sin\big(\frac{\theta-\omega+\pi}\sigma\big)\\
  &=2 \sinh\frac\tau2\Big(e^{-\frac{\tau}{2\sigma}}\sinh(\frac{\tau}{2\sigma})+\sin^2\big(\frac{\theta-\omega+\pi}{2\sigma}\big)\Big)\mp 4 i \cosh\frac\tau2 \sin\big(\frac{\theta-\omega+\pi}{2\sigma}\big)\cos\big(\frac{\theta-\omega+\pi}{2\sigma}\big)\to 0
   \end{align*}
 as $\tau\to 0$ and $\sin^2(\frac{\theta-\omega+\pi}{2\sigma})\to 0$. Therefore, let $b=\big|\sin\big(\frac{\theta-\omega+\pi}{2\sigma}\big)\big|\leq 1$, we control the left hand side of \eqref{est:B11} by
 \begin{align*}
\lesssim_\sigma \int^1_0 \frac{\tau^2+b^2+b}  {\tau^2+b^2}d\tau \lesssim 1+\int^1_0 \frac{b}  {\tau^2+b^2}d\tau\lesssim 1.
 \end{align*}
We emphasize that the implicit constant in the last inequality is independent of $b$, so we have shown \eqref{est:B11}.

Finally, the proof of Proposition \ref{prop:DS} is completed.

\end{proof}

\begin{proof}[The proof of Proposition \ref{prop:LP0} and Proposition \ref{prop:LP}]

By standard methods, the associated Bernstein inequalities and the Littlewood-Paley square function inequalities on the cosmic string spacetime, in Proposition \ref{prop:LP0} and Proposition \ref{prop:LP} respectively are the direct consequences of  the Gaussian upper bounds of heat kernel
\begin{align}\label{est:heat}
 |e^{-tH_{\sigma}^\pm}(x,y)|\lesssim t^{-1}e^{-\frac{c d_g(x,y)^2}{t}}, \forall \, t>0,
 \end{align}
 where $c$ is a small constant independent of $x,y$,  and $d_g(x,y)=\sqrt{r^2+s^2-2rs\cos(d_h(\theta,\omega))}$ denotes the distance between $x=(r,\theta)$ and $y=(s,\omega)$ on the cone $X=(0,+\infty)\times \mathbb{S}^1_{\sigma}$.
 Here $d_h$ is the distance function on $\mathbb{S}^1_{\sigma}$, and $d_h(\theta,\omega)=|\theta-\omega+2\pi j\sigma|$ for any $j\in\Z$.
We refer \cite{A, CDYZh,WZZ2} for more details of the proof of Bernstein inequalities and \cite{BFHM, Zhang1, ZZ} for the Littlewood--Paley inequalities.
The proof of \eqref{est:heat} can be modified from the proof of Proposition \ref{prop:DS}.
Replacing $e^{-it\rho^2}$ by $e^{-t\rho^2}$ in \eqref{KS} and  modifying \eqref{eq:K-bessel}, we obtain the heat kernel
  \begin{align}\label{equ:heat}
  e^{-tH_{\sigma}^\pm}(x,y)= e^{-tH_{\sigma}^\pm}(t, r,\theta,s,\omega)
    =\frac1{2\pi\sigma} \frac{e^{-\frac{r^2+s^2}{4t}}}{2t}  \sum_{k\in\Z} e^{-i\frac k\sigma (\theta-\omega)}  I_{\nu_{\pm}(k)}\Big(\frac{rs}{2t}\Big).
  \end{align}
  To prove \eqref{est:heat}, it suffices to show
 \begin{equation}\label{eq:h1}
e^{-\frac{r^2+s^2}{4t}} \Big| \sum_{k\in\Z} e^{-i\frac k\sigma (\theta-\omega)}  \int_0^{\pi} e^{\frac{rs}{2t}\cos \tau} \cos(\nu_\pm(k) \tau) d\tau \Big| \lesssim e^{-\frac{cd_g(x,y)^2}{t}},
  \end{equation}
  and
  \begin{equation}\label{eq:h2}
  e^{-\frac{r^2+s^2}{4t}} \Big| \sum_{k\in\Z} e^{-i\frac k\sigma (\theta-\omega)}  \sin(\nu_{\pm}(k) \pi)\int_0^\infty e^{-\frac{rs}{2t}\cosh \tau} e^{-\tau\nu_\pm(k)} d\tau \Big| \lesssim  e^{-\frac{cd_g(x,y)^2}{t}}.
  \end{equation}
Similarly as arguing \eqref{est:I1},
From \eqref{metric}, we note that $\theta,\omega\in [-\sigma\pi,\sigma\pi)$ implies $\theta-\omega\in (-2\pi\sigma,2\pi\sigma)$. Define the set
$$A=\{j\in\Z:0\leq|\theta-\omega+2j\pi\sigma|\leq\pi\},$$ then the number of elements of $A$ is no more than $2/\sigma$. Therefore, we obtain
 \begin{equation*}
 \begin{split}
 & \sum_{\{j\in A\}}  e^{\frac{rs}{2t}\cos (\theta-\omega+2j\pi \sigma)} \lesssim_{\sigma}  e^{\frac{rs}{2t}\cos d_h(\theta,\omega)}. \end{split}
\end{equation*}
Finally, we see that
 \begin{equation*}
 \begin{split}
& \text{LHS of }\,\eqref{eq:h1}\lesssim e^{-\frac{r^2+s^2-2rs\cos d_h(\theta,\omega)}{4t}}   \lesssim e^{-\frac{ d_g(x,y)^2}{4t}}
\end{split}
\end{equation*}
and
 \begin{equation*}
 \begin{split}
 \text{LHS of }\,\eqref{eq:h2}&\lesssim e^{-\frac{r^2+s^2}{4t}} |\int^\infty_0e^{-\frac{rs}{2t} \cosh \tau} (B_{\pm}(\tau,\theta,\omega,\sigma)+e^{-\frac\tau2})d\tau\\
 &\lesssim e^{-\frac{r^2+s^2}{4t}} e^{-\frac{rs}{2t}} \int^\infty_0(|B_{\pm}(\tau,\theta,\omega,\sigma)|+e^{-\frac\tau2})d\tau\\
 &\lesssim e^{-\frac{d_g(x,y)^2}{4t}}.
 \end{split}
 \end{equation*}

Summing up, we conclude the proof.
\end{proof}

\subsection{Proof of dispersive estimate for $k=0$}
In this subsection, we concentrate on the case $k=0$, which involves the modified Bessel functions $K_\nu(r)$, unbounded near 0. In order to solve this issue, we will modify the previous construction at $k=0$.

\underline{{\bf Proof of \eqref{est:P0}.}} Since any element of the range of $P_{0}$ can be written as \eqref{equ:f0} below, estimate \eqref{est:P0} is an
immediate consequence of the following result:
\begin{proposition}\label{prop:l-disper0}
  Let $\varphi\in\mathcal{C}_c^\infty([1,2])$, $\tilde{\varphi}\in\mathcal{C}_c^\infty([1/2,4])$ with $\tilde{\varphi}=1$ on the support of $\varphi$ and let
  for some $c\in \mathbb{C}$ and $(\phi_{0},\psi_{0})\in Y$
  \begin{equation}\label{equ:f0}
  P_0f=c\begin{pmatrix}
    \cos \gamma \cdot K_{\frac12}(r) \\
    \sin \gamma \cdot K_{\frac12}(r)
  \end{pmatrix}
  +\begin{pmatrix}
     \phi_{0}(r) \\
     \psi_{0}(r)
  \end{pmatrix}.
  \end{equation}
  Then there exists a constant $C$ such that for all
  $j\in \mathbb{Z}$ one has
  \begin{equation}\label{est:dis0-1}
  \begin{split}
  \Big\|\big(1+&(2^j |x|)^{-\frac12}\big)^{-1}\varphi(2^{-j}|{\mathcal{D}}_{\sigma,\gamma}|) e^{it\mathcal{D}_{\sigma,\gamma}}P_0f(x)\Big\|_{[L^\infty(X)]^2}
  \\&\leq C2^{2j}(1+2^{j}|t|)^{-\frac12} \|\big(1+(2^j |x|)^{-\frac12}\big)\tilde{\varphi}(2^{-j}|{\mathcal{D}}_{\sigma,\gamma}|) P_0f\|_{[L^1{(X)]^2}}.
  \end{split}
  \end{equation}
\end{proposition}

In order to prove Proposition \ref{prop:l-disper0}, we recall the following two key lemmas about the properties of Bessel function. We refer \cite[Lemma 5.1, Lemma 5.2]{CDYZh} for details of proof.

\begin{lemma}\label{lem:bessel}
  Let $\nu_{-}=\max\{0,-\nu\}$.
  For all $x,\nu\in \mathbb{R}$ we have:
  \begin{equation}\label{eq:bess1}
    |J_{\nu}(x)|\le c_{\nu}|x|^{\nu}\bra{x}^{-\nu-\frac 12},
    \qquad
    |J_{\nu}(x)|\le c_{\nu}(1+|x|^{-\nu_{-}}),
  \end{equation}
  \begin{equation}\label{eq:bess2}
    |J'_{\nu}(x)|=
    |J_{\nu-1}(x)-\nu J_{\nu}(x)/x|
    \le c_{\nu}|x|^{\nu-1}\bra{x}^{-\nu+\frac 12}.
  \end{equation}
  Moreover we can write
  \begin{equation}\label{eq:bess3}
    J_{\nu}(x)=x^{-1/2}(e^{ix}a_{+}(x)+e^{-ix}a_{-}(x))
  \end{equation}
  for two functions $a_{\pm}$ depending on $\nu,x$
  and satisfying for all $N\ge1$ and $|x|\ge1$
  \begin{equation}\label{eq:bess4}
    |a_{\pm}(x)|\le c_{\nu,0},
    \qquad
    |\partial_{x}^{N}a_{\pm}(x)|\le c_{\nu,N}
    |x|^{-N-1}.
  \end{equation}
  Here $c_{\nu}$ and $c_{\nu,N}$ are various constants
  depending only on $\nu$ and $\nu,N$ respectively.
\end{lemma}

\begin{lemma}\label{lem:hankelb}
  Let $\nu\ge-1/2$, $\nu_{-}=\max\{0,-\nu\}$,
  $\phi\in C_{c}^{\infty}([1,2])$.
  Then the integral
  \begin{equation}\label{eq:intI}
    I_{\nu}(t;r,s)=
    \int_{0}^{\infty}e^{it \rho}
      J_{\nu}(r\rho)J_{\nu}(s\rho) \phi(\rho)\rho d\rho
  \end{equation}
  satisfies the following estimate for all
  $r_{1},r_{2},t>0$:
  \begin{equation}\label{eq:intIest}
    |I_{\nu}(t;r,s)|\le
    c_{\nu}\bra{t}^{-1/2}(1+r^{-\nu_{-}})(1+s^{-\nu_{-}}).
  \end{equation}
\end{lemma}

\begin{proof}[Proof of Proposition \ref{prop:l-disper0}]
Let $f_j=\tilde{\varphi}(2^{-j}|\D_{\sigma,\gamma}|)f$, the kernel of $\varphi(2^{-j}|\D_{\sigma,\gamma}|)e^{it\D_{\sigma,\gamma}}P_0$ for $k=0$ can be written as
 \begin{equation}\label{equ:FE}
 {\bf K}_0^{(j)}(t,r,s)=\begin{cases}
 \begin{pmatrix}F_{\frac12,j}&0\\0&E_{\frac12,j}\end{pmatrix}, &\text{if}\;\sin\gamma=0,\\
 \begin{pmatrix}E_{\frac12,j}&0\\0&F_{\frac12,j}\end{pmatrix}, &\text{if}\;\cos\gamma=0,
 \end{cases}
 \end{equation}
where
 \begin{equation}\label{equ:F-E-j}
 \begin{aligned}
 &F_{\frac12,j}(t;r,s) =\frac1{2\pi\sigma}\int_0^\infty e^{-it\rho}J_{-\frac12}(r\rho)J_{-\frac12}(s\rho) \varphi(2^{-j}\rho)\,\rho d\rho,\\
 &E_{\frac12,j}(t;r,s) =\frac1{2\pi\sigma}\int_0^\infty e^{-it\rho}J_{\frac12}(r\rho)J_{\frac12}(s\rho) \varphi(2^{-j}\rho)\,\rho d\rho.
 \end{aligned}
 \end{equation}
By scaling, one has
\begin{equation*}
  E_{\frac12,j}(t;r,s)=2^{2j}E_{\frac12,0}(2^jt;2^jr, 2^js)
\end{equation*}
\begin{equation*}
  F_{\frac12,j}(t;r,s)=2^{2j}F_{\frac12,0}(2^jt;2^jr, 2^js).
\end{equation*}
Replacing $t', r', s'$ with $2^{j}t, 2^{j}r, 2^{j}s$ respectively, we obtain two special cases of \eqref{eq:intIest}: $E_{\frac12,0}=I_{\frac12}$ and $F_{\frac12,0}=I_{-\frac12}$
\begin{equation*}
  |E_{\frac12,0}(t,r,s)|\le
  C \bra{t}^{-1/2},
\end{equation*}
\begin{equation*}
  |F_{\frac12,0}(t,r,r)|\le
  C \bra{t}^{-1/2}(1+r^{-\frac12})(1+s^{-\frac12}).
\end{equation*}
Then, we can get
 \begin{equation}\label{est:F-E-j}
\begin{split}
\big|E_{\frac12,j}(t;r,s)\big|&\leq C 2^{2j}
  (1+2^j|t|)^{-1/2},\\
\big|F_{\frac12,j}(t;r,s)\big|&\leq C 2^{2j}
  (1+2^j|t|)^{-1/2}\big(1+(2^j r)^{-\frac12}\big)
  \big(1+(2^j s)^{-\frac12}\big).
\end{split}
\end{equation}
Define the operator
 \begin{equation}\label{T_j}
 \begin{split}
 T^F_jf(t,r)&=\int_0^\infty F_{\frac12,j}(t;r,s) f(s)\, sds,\\  T^E_jf(t,r)&=\int_0^\infty E_{\frac12,j}(t;r,s) f(s)\, sds,
 \end{split}
 \end{equation}
where $f$ is a scalar radial function on $X=C(\Ss)$.
From \eqref{est:F-E-j} we obtain immediately
 \begin{equation}\label{eq:TEj}
\Big\|T^E_jf(x)\Big\|_{L^\infty(X)}
\leq C2^{2j}(1+2^j|t|)^{-1/2} \|f\|_{L^1{(X)}},
\end{equation}
\begin{equation}\label{eq:TFj}
\Big\|\big(1+(2^j r)^{-\frac12}\big)^{-1}T^F_jf(x)\Big\|_{L^\infty(X)}\leq C2^{2j}(1+2^j|t|)^{-1/2} \|\big(1+(2^j r)^{-\frac12}\big)f\|_{L^1{(X)}}.
\end{equation}
Recalling \eqref{equ:FE} and \eqref{equ:F-E-j}, this shows we conclude the proof of
\eqref{est:dis0-1}.
\end{proof}
Therefore, we have shown \eqref{est:P0}.\vspace{0.2cm}

\underline{{\bf Proof of \eqref{est:P0q}.}}
Similar to above, recall $f_j=\tilde{\varphi}(2^{-j}|\D_{\sigma,\gamma}|)f$ and
 \begin{equation*}
 \varphi(2^{-j}|{\mathcal{D}}_{\sigma,\gamma}|)e^{it\mathcal{D}_{\sigma,\gamma}}
 P_0f=\int_{0}^\infty \int_{-\pi}^{\pi}
 \Phi_0(\theta) {\bf K}_{0}^{(j)}(t,r,s)\overline{\Phi_0(\omega)}.
 f_j(s,\omega)  d\omega \;sds
 \end{equation*}
To overcome the difficulty due to Bessel functions of negative order in \eqref{equ:F-E-j}, we use again the scaling argument to reduce the problem in the case $j=0$. Then it suffices to prove the following lemma.
\begin{lemma}\label{lem:Tnuqq'}
  Let $\nu=\pm1/2$ and $T_\nu$ be the operator of the form
  \begin{equation}\label{Tnu-operator}
\begin{split}
(T_{\nu}g)(t,r)=\int_0^\infty  I_{\nu}(t;r,s) g(s)\, sds
 \end{split}
\end{equation}
where $I_{\nu}(t;r,s)$ is defined as \eqref{eq:intI}. Then
 for $2\leq q<4$ we have
  \begin{equation}\label{est:q-q'}
  \|T_{\nu}g\|_{L^q_{rdr}}\le
    c_{\nu}(1+|t|)^{-\frac12(1-\frac 2q)}\|g\|_{L^{q'}_{sds}}.
  \end{equation}
\end{lemma}
\begin{proof} We only consider the case $\nu=-1/2$ since the same argument works for $\nu=1/2$.
Let $\chi\in \mathcal{C}_c^\infty ([0,+\infty)$ be defined as
\begin{equation}
\chi(r)=
\begin{cases}1,\quad r\in [0, \frac12],\\
0, \quad r\in [1,+\infty)
\end{cases},\quad \chi^c(r)=1-\chi.
\end{equation}
Then we split the kernel $I_{\nu}(t;r,s)$ as follows:
  \begin{equation}
\begin{split}
 I_{\nu}(t;r,s)=&\chi(r)I_{\nu}(t;r,s)\chi(s)+\chi^c(r)I_{\nu}(t;r,s) \chi(s)\\
&+\chi(r)I_{\nu}(t;r,s)\chi^c(s)+\chi^c(r)I_{\nu}(t;r,s)\chi^c(s)\\
&:=T^1_{\nu}+T^2_{\nu}+T^3_{\nu}+T^4_{\nu}.
 \end{split}
\end{equation}
We estimate the norms $\|T^j_{\nu}g\|_{L^q_{rdr}}$ for $j=1,2,3,4$ separately.

Firstly, for $T^1_\nu$, by Lemma \ref{lem:hankelb}, we have
  \begin{equation*}
\begin{split}
|\chi(r)I_{\nu}(t;r,s)\chi(s)|\leq (1+|t|)^{-1}\chi(r)r^{-\frac12}\chi(s)s^{-\frac12},
 \end{split}
\end{equation*}
by H\"older inequality, as long as $2\leq q<4$, we obtain
  \begin{equation}\label{est:q-q'1}
  \|T^1_{\nu}g\|_{L^q_{rdr}}\le
    c_{\nu}(1+|t|)^{-1}\Big(\int_0^1 r^{-\frac12 q} r dr\Big)^{2/q}\|g\|_{L^{q'}_{sds}}\le
    c_{\nu}(1+|t|)^{-1}\|g\|_{L^{q'}_{sds}}.
  \end{equation}

In view of $T^2_{\nu}$ and $T^3_{\nu}$ are similar, it suffices to estimate $T^3_{\nu}$. We modified the technique used in the proof of Lemma \ref{lem:hankelb} for the Case $r\le 1$, $s\ge 1$ (see \cite[Section 5]{CDYZh}).
By \eqref{eq:bess3} it reduces to estimate the two integrals
  \begin{equation*}
    I_{\pm}=\int_{0}^{\infty}
    \rho \phi(\rho)J_{\nu}(r\rho)(s\rho)^{-1/2}
    e^{i(t\pm s)\rho}a_{\pm}(s\rho)d \rho.
  \end{equation*}
 If $|t|\leq 1$, integrating by parts we get
  \begin{equation*}
   | I_{\pm}|=\Big|\frac{is^{-1/2}}{\pm s}
    \int_{0}^{\infty}
    (\rho^{1/2}\phi(\rho)J_{\nu}(r\rho)a_{\pm}(s\rho)e^{it\rho})'
    e^{\pm i \rho s}d\rho\Big|\leq Cr^{-\frac12}s^{-\frac32} .
  \end{equation*}
 Hence, for $|t|\leq 1$ and $2\leq q<4$ there holds
      \begin{equation}
  \|T^3_{\nu}g\|_{L^q_{rdr}}\le
    c_{\nu}\Big(\int_0^1 r^{-\frac12 q} r dr\Big)^{1/q}\Big(\int_{\frac12}^{+\infty} s^{-\frac{3q}2} s ds\Big)^{1/q}\|g\|_{L^{q'}_{sds}}\le
    c_{\nu}\|g\|_{L^{q'}_{sds}}.
  \end{equation}
Now we consider the case $|t|\geq 1$. For $j\in\Z$,  if $|t\pm s|\leq 2^j|t|\leq 1$, we have
  \begin{equation*}
    |I_{\pm}|\le C
    (1+r^{-\frac12})s^{-1/2}.
  \end{equation*}
For $j\in\Z$,  if $|t\pm r|\sim 2^j|t|\geq 1$, using by integration by parts we get
 \begin{equation*}
    |I_{\pm}|\le C \frac{r^{-\frac12}s^{-1/2}}
      {|t\pm s|}.
  \end{equation*}
Therefore, we arrive at
  \begin{equation}\label{est:q-q'2}
    \begin{split}
  &\|T^3_{\nu}g\|^q_{L^q_{rdr}}
  =\int_0^1\Big|\sum_{j\in\Z}\int_{2^{j-1}|t|\leq|t\pm r_2|\leq 2^{j}|t|}\chi(r_1)I_{\nu}(t;r,s)\chi^c(s) g(s) s\, ds\Big|^q r\, dr\\
   &\leq\int_0^1\Big(\sum_{\{j\in\Z: 2^j|t|\leq 1/10\}}\int_{2^{j-1}|t|\leq|t\pm s|\leq 2^{j}|t|} r^{-\frac12}s^{-1/2}| g(s)| s\, ds \\&\quad+\sum_{\{j\in\Z: 2^j|t|\geq 1/10\}}\int_{2^{j-1}|t|\leq|t\pm s|\leq 2^{j}|t|} r^{-\frac12} s^{-1/2}|t\pm s|^{-1} |g(s)| s\, ds\Big)^q r\, dr.
    \end{split}
  \end{equation}
Observe that $|t|\geq 1$ and $|t\pm s|\leq 1/10$ implies $s\sim |t|$. thus the first term of \eqref{est:q-q'2}
  \begin{equation*}
  \begin{split}
  &\sum_{\{j\in\Z: 2^j|t|\leq 1/10\}}\int_{2^{j-1}|t|\leq|t\pm s|\leq 2^{j}|t|} s^{-1/2}| g(s)| s\, ds\\
  &\quad\leq \Big(\int_{|t\pm s|\leq 1/10} s^{-q/2} s\, ds\Big)^{\frac1q} \|g\|_{L^{q'}_{sds}}\leq C|t|^{\frac1q-\frac12}\|g\|_{L^{q'}_{sds}}.
  \end{split}
  \end{equation*}
For the second term we have
  \begin{equation*}
  \begin{split}
  &\sum_{\{j\in\Z: 2^j|t|\geq 1/10\}}\int_{2^{j-1}|t|\leq|t\pm s|\leq 2^{j}|t|} s^{-1/2}|t\pm s|^{-1}| g(s)| s\, ds\\
  &\leq \sum_{\{j\in\Z: 2^j|t|\geq 1/10\}}(2^j|t|)^{-1}\int_{2^{j-1}|t|\leq|t\pm s|\leq 2^{j}|t|} s^{-1/2}| g(s)| s\, ds
  \end{split}
  \end{equation*}
and by H\"older inequality
 \begin{equation*}
 \leq \Big(\sum_{\{j\in\Z: 2^j|t|\geq 1/10\}}(2^j|t|)^{-q} \int_{2^{j-1}|t|\leq|t\pm s|\leq 2^{j}|t|} s^{1-\frac q2} \, ds\Big)^{\frac1q} \|g\|_{L^{q'}_{sds}}.
 \end{equation*}
Since $\frac1{20}\leq 2^{j-1}|t|\leq|t\pm s|\leq 2^{j}|t|$ implies that $|t|\leq s\leq 2^{j+1}|t|$ for $j\geq 0$, one has
   \begin{equation*}
    \begin{split}
&\sum_{\{j\in\Z: 2^j|t|\geq 1/10\}}(2^j|t|)^{-q} \int_{2^{j-1}|t|\leq|t\pm s|\leq 2^{j}|t|} s^{1-\frac q2} \, ds\\
&\leq \sum_{\{j\geq 0: 2^j|t|\geq 1/10\}}(2^j|t|)^{-q} \max\{|t|^{2-\frac q2}, (2^j|t|)^{2-\frac q2}\}\leq C|t|^{2-\frac{3q}2}.
    \end{split}
  \end{equation*}
Thus for $|t|\geq1$ it holds
  \begin{equation*}
    \begin{split}
  &\|T^3_{\nu}g\|_{L^q_{rdr}}\leq C \big(|t|^{\frac1q-\frac{1}2}+|t|^{\frac2q-\frac{3}2}\big)\int_0^1 r^{-\frac12 q} rdr \|g\|_{L^{q'}_{sds}}\leq C |t|^{\frac1q-\frac{1}2}\|g\|_{L^{q'}_{sds}}.
    \end{split}
  \end{equation*}
Summing up, for all $|t|$ we get
 \begin{equation}
 \|T^3_{\nu}g\|_{L^q_{rdr}}\leq C (1+|t|)^{\frac1q-\frac{1}2}\|g\|_{L^{q'}_{sds}}.
 \end{equation}
We finally deal with $T^4_{\nu}$. By Lemma \ref{lem:hankelb}, one has
 \begin{equation*}
  \|T^4_{\nu}g\|_{L^\infty_{rdr}}\le
    c_{\nu}(1+|t|)^{-\frac12}\|g\|_{L^{1}_{sds}}.
  \end{equation*}
Combine with the fact the Hankel transform is unitary on $L^2_{rdr}$, we have
     \begin{equation*}
  \|T^4_{\nu}g\|_{L^2_{rdr}}\le   \|\mathcal{H}_\nu e^{it\rho}\phi(\rho) \mathcal{H}_\nu e^{it\rho} g\|_{L^2_{r dr}}\leq C \|g\|_{L^{2}_{sds}}.
  \end{equation*}
Using the interpolation theorem, we obtain
      \begin{equation}\label{est:q-q'4}
    \begin{split}
  &\|T^4_{\nu}g\|_{L^q_{rdr}}\leq C (1+|t|)^{-\frac12(1-\frac2q)}\|g\|_{L^{q'}_{sds}}.
    \end{split}
  \end{equation}
In conclusion we obtain \eqref{est:q-q'} from the above estimates and conclude the proof.
\end{proof}

\section{Strichartz estimates}\label{sec:stri}
In this section, we focus on the proof of Strichartz estimates for the Dirac equation in the cosmic string space-time. We deal with the two components $P_{\bot}f$ and $P_0f$ separately as before.
We shall prove Theorems \ref{th-stri1} and \ref{th-stri3} using a
variant of the abstract Keel--Tao argument and the dispersive estimates proved in the previous sections.

\subsection{Proof of Theorem \ref{th-stri1}}

Our proof is on the basic of the abstract Keel-Tao argument \cite{KT}, we give a invariant version proved in \cite{Zhang2}:

\begin{proposition}\label{prop:semi}
Let $(X,\mathcal{M},\mu)$ be a $\eta$-finite measured space and
$U:\mathbb{R}\rightarrow B(L^2(X,\mathcal{M},\mu))$ be a weakly
measurable map satisfying, for some constants $C$, $\kappa\geq0$,
$\eta, h>0$,
\begin{equation}\label{md}
\begin{split}
\|U(t)\|_{L^2\rightarrow L^2}&\leq C,\quad t\in \mathbb{R},\\
\|U(t)U(\tau)^*f\|_{L^\infty}&\leq
Ch^{-\kappa}(h+|t-\tau|)^{-\eta}\|f\|_{L^1}.
\end{split}
\end{equation}
Then for every pair $q,p\in[2,\infty]$ such that $(p,q,\eta)\neq
(2,\infty,1)$ and
\begin{equation*}
\frac{1}{p}+\frac{\eta}{q}\leq\frac\eta 2,\quad p\ge2,
\end{equation*}
there exists a constant $\tilde{C}$ only depending on $C$, $\eta$,
$q$ and $p$ such that
\begin{equation*}
\Big(\int_{\mathbb{R}}\|U(t) u_0\|_{L^q}^p dt\Big)^{\frac1p}\leq \tilde{C}
\Lambda(h)\|u_0\|_{L^2}
\end{equation*}
where $\Lambda(h)=h^{-(\kappa+\eta)(\frac12-\frac1q)+\frac1p}$.
\end{proposition}

Recalling \eqref{m:proH}
 \begin{align*}
 e^{it\D_{\sigma,\gamma}}P_{\bot}=e^{it\sqrt{H}}P_{\bot},
 \end{align*}
the estimate \eqref{stri-D1} follows from
 \begin{equation}\label{stri-D1-1}
 \|e^{-it \sqrt{H}}f\|_{[L^p_t(\R; L^q_x(X))]^2}\leq
 C\| f\|_{\dot{\mathbf{H}}^{s}}.
 \end{equation}
with $p,q, s$ as in Theorem \ref{th-stri1}.

Let $f_j=\varphi(2^{-j}\sqrt{H})f$. As a consequence of the square function estimates \eqref{est:LP} and the Minkowski inequality, we obtain
 \begin{equation}\label{LP}
 \|e^{-it\sqrt{H}}f\|_{[L^p(\R;L^q(X))]^2}\lesssim
 \Big(\sum_{j\in\Z}\|e^{-it\sqrt{H}}f_j\|^2_{[L^p(\R;L^q(X))]^2}\Big)^{\frac12}.
 \end{equation}
Taking cutoffs $\tilde{\varphi}$ satisfying  $\tilde{\varphi}\varphi=\varphi$, we write the propagator as
$$e^{-it \sqrt{H} }f_j=
U_{j}(t)\tilde{f}_{j},
\qquad
U_{j}(t):=\varphi(2^{-j}\sqrt{H})e^{-it \sqrt{H} },
\qquad
\tilde{f}_j=\tilde{\varphi}(2^{-j}\sqrt{H})f.$$
 Let $g=\tilde{f}_j$, by the spectral theorem we know $U_j(t)$ is bounded on $L^2$:
\begin{equation*}
\|U_j(t)g\|_{[L^2(X)]^2}\leq C\|g\|_{[L^2(X)]^2}.
\end{equation*}
On the other hand, the dispersive estimate \eqref{est:Pk} yields
\begin{equation*}
\|U_j(t)U_j^*(\tau)g\|_{[L^\infty (X)]^2}=\|U_j(t-\tau)g\|_{[L^\infty (X)]^2}\leq C 2^{\frac32 j}\big(2^{-j}+|t-\tau|\big)^{-\frac12}\|g\|_{[L^1(X)]^2}.
\end{equation*}
Then assumptions \eqref{md} are valid
for $U_{j}(t)$ with $\kappa=3/2$, $\eta=1/2$ and
$h=2^{-j}$, and applying the Proposition \ref{prop:semi} we obtain
\begin{equation*}
\|U_j(t)g\|_{[L^p(I;L^q(X))]^2}\leq C
2^{j[2(\frac12-\frac1p)-\frac1q]} \|g\|_{[L^2(X)]^2}
\end{equation*}
which implies \eqref{stri-D1-1}
\begin{equation}\label{LP}
\|e^{-it\sqrt{H}}f\|_{[L^p(\R;L^q(X))]^2}\lesssim
\Big(\sum_{j\in\Z}2^{2js} \|\tilde{\varphi}_j(\sqrt{H})f\|^2_{[L^2(X)]^2}\Big)^{\frac12}
\end{equation}
where $s=2(\frac12-\frac1q)-\frac1p$. Therefore,  we finish the proof of Theorem \ref{th-stri1}.

\subsection{Proof of Theorem \ref{th-stri3}}\label{subs-th3}
We first prove the Strichartz estimates \eqref{stri-D3} when $2\leq q<4$ and $\frac2p+\frac1q\leq \frac12$ by using Keel-Tao's argument. Then, by proving \eqref{stri-D3} at $(p,q)=(4+2\epsilon,4-\epsilon)$ with $0<\epsilon\ll1$ (which is close to the corner point $F (4,4)$) and by interpolation, we extend $\frac2p+\frac1q\leq \frac12$ to $\frac1p+\frac1q < \frac12$.

In order to achieve the first goal we shall need a second variant of Keel--Tao's argument, proved in \cite{CDYZh}.
This is because that for the radial component, we only have weightless dispersive estimates \eqref{est:P0q} for $2\leq q<4$.

\begin{proposition}\label{prop:semi-1}
Let $(X,\mathcal{M},\mu)$ be a $\sigma$-finite measured space and
$U: \mathbb{R}\rightarrow B(L^2(X,\mathcal{M},\mu))$ be a weakly
measurable map satisfying, for some constants $C$, $\kappa\geq0$,
$\eta, h>0$,
\begin{equation}\label{md-1}
\begin{split}
\|U(t)\|_{L^2\rightarrow L^2}&\leq C,\quad t\in \mathbb{R},\\
\|U(t)U(\tau)^*f\|_{L^{q_0}}&\leq
Ch^{-\kappa(1-\frac2{q_0})}(h+|t-\tau|)^{-\eta(1-\frac2{q_0})}\|f\|_{L^{q_0'}},\quad 2\leq q_0\leq +\infty.
\end{split}
\end{equation}
Then for every pair $q,p\in[1,\infty]$ such that $(p, q,\eta)\neq
(2,\infty,1)$ and
\begin{equation*}
\frac{1}{p}+\frac{\eta}{q}\leq\frac\eta 2,\quad 2\leq q\leq q_0, p\geq 2.
\end{equation*}
there exists a constant $\tilde{C}$ depending only on $C$, $\eta$,
$q$ and $p$ such that
\begin{equation*}
\Big(\int_{\mathbb{R}}\|U(t) u_0\|_{L^q}^p dt\Big)^{\frac1p}\leq \tilde{C}
\Lambda(h)\|u_0\|_{L^2}
\end{equation*}
where $\Lambda(h)=h^{-(\kappa+\eta)(\frac12-\frac1q)+\frac1p}$.
\end{proposition}
To use the frequency-localized dispersive estimates to prove the Strichartz estimates,
we thus require the Littlewood-Paley theory for the radial component in order to eliminate the frequency localization. Usually the square function estimates can be obtained by the Gaussian upper bounds of the heat kernel. However, the singularity of Bessel functions of negative order prevents us from proving the Gaussian upper bounds for $e^{-t\D^2_{\sigma,\gamma}}P_0$. We thereby have to directly establish the the multiplier estimates in $L^{p}$ for the Littlewood-Paley operator in the case $\frac43<p<4$, which is inspired by the previous analysis in Lemma \ref{lem:Tnuqq'}. For our purpose, it is enough to prove the Strichartz estimates with $2\leq q<4$ because of $[2,4)\subset(\frac43,4)$.


\begin{proposition}[Square function inequality]
  \label{prop:squarefun0}
  Let $\varphi\in C_c^\infty$, with support in $[1/2,1]$
  and such that $0\leq\varphi\leq 1$ and
  \begin{equation}\label{LP-dp}
    \sum_{j\in\Z}\varphi(2^{-j}\lambda)=1
    \qquad\text{where}\qquad
    \varphi_j(\lambda):=\varphi(2^{-j}\lambda).
  \end{equation}
  Then for all $\frac34<p<4$
  there exist constants $c_p$ and $C_p$ depending on $p$ such that
  \begin{equation}\label{square1-0}
  c_p\|P_0 f\|_{[L^p(X)]^2}\leq
  \Big\|\Big(\sum_{j\in\Z}|\varphi_j(|\mathcal{D}_{A,\gamma}|) P_0f|^2\Big)^{\frac12}\Big\|_{[L^p(X)]^2}\leq
  C_p\|P_0f\|_{[L^p(X)]^2}.
  \end{equation}
\end{proposition}

\begin{remark}We stress that the square function inequality \eqref{square1-0} is valid only for $\frac34<p<4$ which is different from the previous one \eqref{est:LP}.
\end{remark}

\begin{proof} Using by the Rademacher functions and Stein's argument in \cite[Appendix D]{Stein},
we need to prove the Littlewood-Paley operator $\varphi_j(|\mathcal{D}_{\sigma,\gamma}|)P_0$ is bounded on $L^p(X)$ in the range $\frac43<p<4$. We only consider the case $\nu=-\frac12$, since the case $\nu=\frac12$ can be proved in a similar way and we omit the details.

Comparing with \eqref{equ:FE} and \eqref{equ:F-E-j}, we have
\begin{equation*}
  \varphi_j(|\mathcal{D}_{\sigma,\gamma}|)P_0 f=
  \int_{0}^{\infty}K_0(r,s)f(s)sds
\end{equation*}
where the kernel $K_0(r,s)$ can be explicitly written as
 \begin{equation}
 K_0(r,s)=\begin{cases}
 \begin{pmatrix}LP_{-\frac12,j}&0\\0&LP_{\frac12,j}\end{pmatrix}, &\text{if}\;\sin\gamma=0,\\
 \begin{pmatrix}LP_{\frac12,j}&0\\0&LP_{-\frac12,j}\end{pmatrix}, &\text{if}\;\cos\gamma=0,
 \end{cases}
 \end{equation}
where
 \begin{align}
 LP_{\nu,j}=\int^\infty_0\varphi_j(\rho)J_\nu(r\rho)J_\nu(s\rho)\rho d\rho,\quad \nu=-\frac12.
 \end{align}
By scaling $LP_{\nu,j}=2^{2j}LP_{\nu,0}$, it suffices to consider $LP_{\nu,0}$. Define the operator $\mathcal{T}_{\nu}$ as
 \begin{align}
 (\mathcal{T}_{\nu}g)(t,r)=\int^\infty_0 LP_{\nu,0}g(s)sds.
 \end{align}
From \eqref{eq:bess1}, one has
 \begin{align}
 |LP_{\nu,0}|=|\int^\infty_0\varphi(\rho)J_\nu(r\rho)J_{\nu}(s\rho)\rho d\rho|\leq C(1+r^{-\nu_-})(1+s^{-\nu_-}).
 \end{align}
Let $\chi\in \mathcal{C}_c^\infty ([0,+\infty)$ be
\begin{equation}
\chi(r)=
\begin{cases}1,\quad r\in [0, \frac12],\\
0, \quad r\in [1,+\infty)
\end{cases},\quad \chi^c(r)=1-\chi.
\end{equation}
Then we split the kernel $I_{\nu}(t;r,s)$ into four terms as follows:
  \begin{equation}
\begin{split}
 K_0(r,s)=&\chi(r)K_0(r,s)\chi(s)+\chi^c(r)K_0(r,s) \chi(s)\\
&+\chi(r)K_0(r,s)\chi^c(s)+\chi^c(r)K_0(r,s)\chi^c(s).
 \end{split}
\end{equation}
We thus estimate the corresponding norms $\|\mathcal{T}^j_{\nu}g\|_{L^p_{rdr}}$ for $j=1,2,3,4$ separately.

\underline{\bf Term $\mathcal{T}^1_{\nu}$}. We have that
\begin{equation}
 \begin{split}
 |\chi(r)K_{0}(r,s)\chi(s)|\leq \chi(r)r^{\nu}\chi(s)s^{\nu};
 \end{split}
\end{equation}
hence, as long as $1<p, p'<-\frac2\nu=4$
we obtain by H\"older inequality
  \begin{equation}\label{est:q-q'1}
  \|\mathcal{T}^1_{\nu}g\|_{L^p_{rdr}}\le
    c_{\nu}\Big(\int_0^1 r^{\nu p} r dr\Big)^{1/p}\Big(\int_0^1 s^{\nu p'} sds\Big)^{1/p'}\|g\|_{L^{p}_{sds}}\le
    c_{\nu}\|g\|_{L^{p}_{sds}}.
  \end{equation}

\underline{\bf Terms $\mathcal{T}^2_{\nu}$ and $\mathcal{T}^3_\nu$}.
Since the terms $\mathcal{T}^2_{\nu}$ and $\mathcal{T}^3_{\nu}$ are similar to $T^2_\nu$ and $T^3_\nu$ in Lemma \ref{lem:Tnuqq'} with $t=0$, we only need the same method for $|t|\leq 1$ in the proof of Lemma \ref{lem:Tnuqq'}. Since $\mathcal{T}^2_{\nu}$ is very similar to $\mathcal{T}^3_\nu$,
we only estimate $\mathcal{T}^{3}_{\nu}$. For the term $\mathcal{T}^3_\nu$,
we need to estimate the two integrals
  \begin{equation*}
    I_{\pm}=\int_{0}^{\infty}
    \rho \phi(\rho)J_{\nu}(r \rho)(s\rho)^{-1/2}
    e^{\pm i s\rho}a_{\pm}(s\rho)d \rho.
  \end{equation*}
Using \eqref{eq:bess2}
and integrating by parts twice we obtain
  \begin{equation*}
   | I_{\pm}|=\Big|\frac{is^{-1/2}}{\pm s^2}
    \int_{0}^{\infty}
    (\rho^{1/2}\phi(\rho)J_{\nu}(r\rho)a_{\pm}(s\rho))'
    e^{\pm i \rho s}d\rho\Big|\leq
    Cr^{\nu}s^{-\frac52} .
  \end{equation*}
Then as long as $1<p<-\frac2\nu=4$,  we have
 \begin{equation}
  \|\mathcal{T}^3_{\nu}g\|_{L^p_{rdr}}\le
    c_{\nu}\Big(\int_0^1 r^{\nu p} r dr\Big)^{1/p}\Big(\int_{\frac12}^{+\infty} s^{-\frac{5p'}2} s ds\Big)^{1/p'}\|g\|_{L^{p}_{sds}}\le
    c_{\nu}\|g\|_{L^{p}_{sds}}.
 \end{equation}
Similarly, the term $\mathcal{T}^2_{\nu}$
is bounded on $L^p_{rdr}$ provided
$1<p'<-\frac2\nu$.

\underline{\bf Term $\mathcal{T}^4_{\nu}$}.
We examine the two integrals
\begin{equation*}
  I_{\pm}(r,s)=\int_{0}^{\infty}
  \rho \phi(\rho)(r \rho)^{-1/2}(s\rho)^{-1/2}
  e^{ i(r \pm s)\rho}a_{\pm}(r \rho) a_{\pm}(s\rho)d \rho.
\end{equation*}
Integrating by parts, for $r,s\geq 1/2$ and any $N\geq 0$, we obtain
  \begin{equation*}
   |  I_{\pm}(r,s)|\leq C_N (rs)^{-\frac12} (1+|r-s|)^{-N}.
  \end{equation*}
In order to prove $\mathcal{T}^4_{\nu}$ is bounded on $L^p_{rdr}$ via the Schur test, we must prove that
\begin{equation*}
 \sup_{s\in [\frac12,+\infty)}
 \int_{\frac12}^\infty
 |  I_{\pm}(r,s)| rdr\leq C,\quad  \sup_{r \in [\frac12,+\infty)} \int_{\frac12}^\infty |  I_{\pm}(r,s)| sds\leq C.
\end{equation*}
Obviously, it is sufficient to deal with the first estimate.
In view of $r^{1/2}\lesssim |r-s|^{1/2}+s^{1/2}$,
we can write
\begin{equation*}
  |I_{\pm}(r,s)|r \le
  Cs^{-1/2}(1+|r-s|)^{-N+1/2}+
  C(1+|r-s|)^{-N}
\end{equation*}
and the fact that $s\ge1/2$ implies
\begin{equation*}
  \int_{\frac12}^\infty
  |  I_{\pm}(r-s)| rdr\le
  C\int_{\mathbb{R}}(1+|r-s|)^{-N'}dr \le C.
\end{equation*}
We conclude that $\mathcal{T}^4_{\nu}$ is bounded on $L^p_{rdr}$ for $1<p<\infty$.

This concludes the proof for $\frac43<p<4$.
\end{proof}

Combining the dispersive estimate \eqref{est:P0q}, Keel--Tao's theorem and the Square function inequality, we obtain the Strichartz estimates with $2\leq q<4$ and $\frac2p+\frac1q\leq\frac12$. To extend the range $\frac2p+\frac1q\leq\frac12$ to $\frac1p+\frac1q<\frac12$, by interpolation, it suffices to prove that there exists a point $(p_0,q_0)$ arbitrarily close to the point $F$ of Figure one such that
\begin{equation}\label{est:44}
  \| e^{it\mathcal{D}_{\sigma,\gamma}}P_{0}f\|
  _{[L^{p_0}_t(\R; L^{q_0}_x(X))]^2}
  \leq C \|P_{0} f\|_{\dot{\mathbf{H}}^{s}(X)},
  \quad s=1-\frac1{p_0}-\frac2{q_0}.
\end{equation}
If this could be done, the proof of Theorem \ref{th-stri3} is concluded.
We recall
\begin{align*}
 P_0f=\begin{pmatrix}\phi_0(r)\\ \psi_0(r)\end{pmatrix}=c\begin{pmatrix}
    \cos \gamma \cdot K_{\frac12}(r) \\
    \sin \gamma \cdot K_{\frac12}(r)
  \end{pmatrix}
  +\begin{pmatrix}\tilde{\phi}_0(r)\\ \tilde{\psi}_0(r)\end{pmatrix}.
 \end{align*}
In view of propagator kernel \eqref{equ:FE}, then we deal with the two cases:
\begin{itemize}
\item if $\sin\gamma=0$, then
 \begin{align*}
 e^{it\D_{\sigma,\gamma}}P_0f=\int^\infty_0\int^\pi_{-\pi}\begin{pmatrix} F_{\frac12}&0\\0&E_{\frac12}\end{pmatrix}\begin{pmatrix}\phi_0(r)\\ \psi_0(r)\end{pmatrix}d\omega sds;
 \end{align*}
\item if $\cos\gamma=0$, then
 \begin{align*}
 e^{it\D_{\sigma,\gamma}}P_0f=\int^\infty_0\int^\pi_{-\pi}\begin{pmatrix} E_{\frac12}&0\\0&F_{\frac12}\end{pmatrix}\begin{pmatrix}\phi_0(r)\\ \psi_0(r)\end{pmatrix}d\omega sds,
 \end{align*}
\end{itemize}
where
  \begin{equation}
 \begin{aligned}
 &F_{\frac12}(t;r,s) =\frac1{2\pi\sigma}\int_0^\infty e^{-it\rho}J_{-\frac12}(r\rho)J_{-\frac12}(s\rho) \,\rho d\rho,\\
 &E_{\frac12}(t;r,s) =\frac1{2\pi\sigma}\int_0^\infty e^{-it\rho}J_{\frac12}(r\rho)J_{\frac12}(s\rho) \,\rho d\rho.
 \end{aligned}
 \end{equation}

Now define the operator as (see \eqref{Tnu-operator})
 \begin{align}\label{op:Tnu}
 (T_\nu g)(t,r)=\int^\infty_0{ K}_\nu(t,r,s)g(s)sds
 \end{align}
where ${ K}_\nu(t,r,s)$ is given by
 \begin{align}\label{Knu}
 { K}_{\nu}(t,r,s)=\frac1{2\pi\sigma}\int^\infty_0e^{-it\rho}J_{\nu}(r\rho)J_{\nu}(s\rho)\rho d\rho,\quad \nu=\pm\frac12.
 \end{align}
To prove the goal estimate, it suffices to prove
 \begin{align}\label{eq:wless}
\|(T_\nu g) (t,r)\|_{L^{p_0}(\R;L^{q_0}_{rdr})}\leq C\|g\|_{\dot{H}^s}, \quad\nu=\pm1/2.
 \end{align}
 Assume $2\leq q_0<4$, by using the square function estimate \eqref{square1-0} and the scaling argument, 
 we are reduce to prove 
  \begin{align}\label{eq:wless'}
\|(T_\nu g) (t,r)\|_{L^{p_0}_t(\R;L^{q_0}_{rdr})}\leq C\|g\|_{L^2_{sds}}, \quad \text{supp}\mathcal{H}_{\nu}g\subset [1,2].
 \end{align}
As before, we only consider the more difficult case $\nu=-1/2$. Let $h=\mathcal{H}_{-1/2}g$, now we write 
 \begin{equation*}
 Z=\int^\infty_0e^{-it\rho}J_{-\frac12}(r\rho)h(\rho)\rho d\rho.
 \end{equation*}
Due to the compact support of $h$, we rewrite
 \begin{equation*}
 \begin{split}
 Z&=\mathcal{F}_{\rho\rightarrow t}\left( r\rho)^{\frac12}J_{-\frac12}(r\rho)\cdot \rho^{\frac12}r^{-\frac12}h(\rho) \right)\\
 &=\mathcal{F}_{\rho\rightarrow t}\left((r\rho)^{\frac12}J_{-\frac12}(r\rho) \right)\ast \mathcal{F}_{\rho\rightarrow t}\left(\rho^{\frac12}h(\rho) \right)\cdot r^{-\frac12}.
 \end{split}
 \end{equation*}
Observe that $$(r\rho)^{\frac12}J_{-\frac12}(r\rho)=c\cos(r\rho)=\frac{c}2(e^{ir\rho}+e^{-ir\rho}),$$ we have 
 \begin{equation*}
 \mathcal{F}_{\rho\rightarrow t}\left((r\rho)^{\frac12}J_{-\frac12}(r\rho) \right)\cong \delta_{0}(t+r)+\delta_{0}(t-r),
 \end{equation*}
thus we obtain
 \begin{equation}\label{equ:Z}
 Z=|x|^{-\frac12}\left(\hat{\tilde{h}}(t+|x|)+\hat{\tilde{h}}(t-|x|)  \right), \quad \tilde{h}=\rho^{\frac12}h(\rho).
 \end{equation}
 where $r=|x|$ and $\hat{\tilde{h}}$ means the Fourier transform of $\tilde{h}$.
Let $\chi(r)\in C^\infty_c(\R)$ with supp$\chi\subset [0,1]$. We split $Z$ as follows
 \begin{equation*}
 Z=Z_1+Z_2,\quad Z_1=Z\chi(|x|).
 \end{equation*}
 Then for any $0<\epsilon\leq \eta\ll 1$, we obtain
  \begin{equation}\label{eq:Z1}
  \begin{split}
 \|Z_1\|_{L^{4+\eta}_t L^{4-\epsilon}_{rdr}}&\leq\|Z_1\|_{ L^{4-\epsilon}_{rdr} L^{4+\eta}_t}\\
 &\leq \|r^{-\frac12}\|_{L^{4-\epsilon}_{rdr}(r\leq1)}\|\hat{\tilde{h}}\|_{L^{4+\eta}_t}\lesssim \|\tilde{h}\|_{L^{(4+\eta)'}_t} \lesssim\|h\|_{L^2_{\rho d\rho}},
 \end{split}
 \end{equation}
 where we use the Hausdorff-Young inequality and the compact support of $h$.
 Next we consider the term $Z_2$.
Since $\mathcal{H}_{\nu}$ is unitary on $L^2_{rdr}$, we know the following  trivial estimates
 \begin{equation}\label{eq:Zt}
\|Z_2\|_{L^\infty_t L^2_{rdr}}\lesssim \|Z\|_{L^\infty_t L^2_{rdr}}\leq C\|h\|_{L^2_{\rho d\rho}}.
 \end{equation}
 For any $0<\varepsilon\ll 1$ and $\eta\geq\varepsilon$, similarly as above, we obtain
  \begin{equation}\label{eq:Z2}
    \begin{split}
 \|Z_2\|_{L^{4+\eta}_t L^{4+\varepsilon}_{rdr}}&\leq\|Z_2\|_{ L^{4+\varepsilon}_{rdr} L^{4+\eta}_t}\\
 &\leq \|r^{-\frac12}\|_{L^{4+\varepsilon}_{rdr}(r\geq1)}\|\hat{\tilde{h}}\|_{L^{(4+\eta)'}_t} \lesssim\|h\|_{L^2_{\rho d\rho}}.
 \end{split}
 \end{equation}
By interpolating with \eqref{eq:Zt} and \eqref{eq:Z2}, we get
 \begin{equation}
 \|Z_2\|_{L^p_t L^{q}_{rdr}}\lesssim \|h\|_{L^2_{\rho d\rho}},\quad (1+\frac{\varepsilon}2)\frac1p+\frac1q\leq\frac12,
 \end{equation}
where $p\in[4+\varepsilon,\infty]$, $q\in [2,4+\varepsilon]$. Therefore, there exists one pair $(p_0,q_0)=(4+2\epsilon,4-\epsilon)$ with $0<\varepsilon\ll \epsilon\ll 1$ such that
$$(1+\frac{\varepsilon}2)\frac1{p_0}+\frac1{q_0}\leq\frac12,$$  hence 
 \begin{equation}
 \|Z_2\|_{L^{p_0}_t L^{q_0}_{rdr}}\lesssim \|h\|_{L^2_{\rho d\rho}}.
 \end{equation}
Taking $\epsilon$ small enough, then we see that $(p_0,q_0)$ is a one point arbitrarily close to $(4,4)$. This together with \eqref{eq:Z1} and the fact that $\mathcal{H}_{\nu}$ is unitary on $L^2_{\rho d\rho}$, we prove 
\eqref{eq:wless'}. Therefore, we finish the proof of \eqref{est:44}.\vspace{0.2cm}

{\bf Comments on Remark \ref{rem:q=4}. } To prove \eqref{stri-D3'}, we can use the same argument as above. However, we don't split  \eqref{equ:Z} and directly estimate 
  \begin{equation}\label{eq:Z}
  \begin{split}
 \|Z\|_{L^{4}_t L^{4,\infty}_{rdr}}
 &\leq \||x|^{-\frac12}\|_{L^{4,\infty}_{x}}\|\hat{\tilde{h}}\|_{L^{4}_t}\lesssim \|\tilde{h}\|_{L^{\frac43}_t} \lesssim\|h\|_{L^2_{\rho d\rho}}.
 \end{split}
 \end{equation}
By interpolating  this with $ \|Z\|_{L^{\infty}_t L^{2}_{rdr}}\leq \|h\|_{L^2_{\rho d\rho}},$ we obtain 
  \begin{equation}\label{eq:Z}
  \begin{split}
 \|Z\|_{L^{p}_t L^{q,\tilde{q}}_{rdr}} \lesssim\|h\|_{L^2_{\rho d\rho}},
 \end{split}
 \end{equation}
 provided
 $$\frac1p+\frac1q\leq \frac12,\quad \frac1{\tilde{q}}=\frac2q-\frac12, \quad p\in[4, +\infty], \, q\in[2,4].$$
 However, $\tilde{q}>q$ when $q\in (2,4]$, the embedding inequality of Lorentz space shows that \eqref{stri-D3'} is weaker than \eqref{stri-D3}.

\subsection{Proof of  Theorem \ref{th-stri2}}
As a first step, we shall prove \eqref{stri-D2b} for $\theta=1$
\begin{equation}\label{stri-D2}
  \|W(|x|) e^{it\mathcal{D}_{\sigma,\gamma}}P_{0}f\|
  _{[L^p_t(\R; L^q_x(X))]^2}
  \leq C \|P_{0} f\|_{\mathbf{H}^{s}(X)},
  \quad s=1-\frac1p-\frac2q,
\end{equation}
for all $(p,q)$ satisfying \eqref{pqrange0}.
This estimate is a consequence of
the following two estimates by complex interpolation
\begin{equation}\label{est:2-2}
\|W(|x|)e^{it\mathcal{D}_{\sigma,\gamma}}P_{0}f\|
  _{[L^\infty_t(\R; L^2_x(X))]^2}\leq
   C \|P_{0} f\|_{[L^2]^2}
\end{equation}
and
\begin{equation}\label{est:inf-2}
\|W(|x|)e^{it\mathcal{D}_{\sigma,\gamma}}P_{0}f\|_{[L^p_t(\R; L^\infty_x(X))]^2}\leq
C
\|P_{0} f\|_{\mathbf{H}^{1-\frac1p}(X)}, \quad \forall\, p>2.
\end{equation}

Then it reduces to prove the following lemma:
\begin{lemma}\label{lemwe}
  Let $p>2$ and $0<\epsilon\ll1$ and $T_{\nu}$ be in \eqref{op:Tnu}.
  We have
\begin{equation}\label{est:2-2'}
\|\omega_{\nu}(r)(T_{\nu}g)(t,r)\|_{L^\infty(\R;L^2_{rdr})}
\leq C\|g\|_{L^2(X)},
\end{equation}
and
\begin{equation}\label{est:inf-2'}
\|\omega_{\nu}(r)(T_{\nu}g)(t,r)\|_{L^p(\R;L^\infty_{rdr})}
\leq C\|g\|_{H^{1-\frac1p}(X)},
\end{equation}
where
\begin{equation}
\omega_{\nu}(r)=
\begin{cases}
    1 &
    \ \text{if}\ \ \nu= 1/2,\\
    (1+r^{-\frac12-\epsilon})^{-1}&
    \ \text{if}\ \ \nu=-1/2.
\end{cases}
\end{equation}
\end{lemma}

\begin{proof}
By the properties of the Hankel transform, we rewrite the operator $T_\nu$
 \begin{equation}
   T_{\nu}g=\mathcal{H}_{\nu}e^{it\rho} \mathcal{H}_{\nu}g.
\end{equation}
Since $|\omega_{\nu}(r)|\leq 1$ and $\mathcal{H}_{\nu}$
is unitary on $L^{2}(rdr)$,
the proof of \eqref{est:2-2'} is trivial.

Denote $R,N\in 2^{\mathbb{Z}}$ are dyadic numbers. Inspired by \cite{CYZ}, we use a dyadic decomposition to write
\begin{eqnarray*}
\|\omega_{\nu}(r)T_{\nu}g(t,r)\|_{L^p_tL^\infty_{dr}}^2&\leq& \left\|\omega_{\nu}(r)\sum_{N\in 2^\Z}\hank\left[e^{it{\rho}} \varphi(\frac{\rho}N)\hank g(\rho) \right](r)\right\|_{L^p_tL^\infty_{dr}(\R^+)\,}^2
\\
&\leq&
\left\|\sup_{R\in2^{\Z}} \left\|\omega_{\nu}(r)\sum_{N\in 2^\Z} \hank\left[e^{it{\rho}} \varphi(\frac{\rho}N)\hank g(\rho) \right]\right\|_{L^\infty_{dr}([R,2R])}\right\|_{L^p_t}^2.
\end{eqnarray*}
Here we replace the measure $rdr$ in \eqref{est:inf-2'} by $dr$ to more intuitively focus on the $L^\infty$-norm.
We obtain by the fact $\ell^2 \hookrightarrow\ell^\infty$ and the Minkowski inequality
\begin{eqnarray*}
\|\omega_{\nu}(r)T_{\nu}g(t,r)\|_{L^p_tL^\infty_{dr}}^2
&\leq&
\left\| \left(\sum_{R\in2^{\Z}}\left\|\omega_{\nu}(r)\sum_{N\in 2^\Z} \hank\left[e^{it{\rho}} \varphi(\frac{\rho}N)\hank g(\rho) \right](r)\right\|^2_{L^\infty_{dr}\,([R,2R])}\right)^{1/2}\right\|_{L^p_t}^2
\\
&\leq&
 \sum_{R\in2^{\Z}}\left\|\omega_{\nu}(r)\sum_{N\in 2^\Z} \hank\left[e^{it{\rho}} \varphi(\frac{\rho}N)\hank g(\rho) \right]\right\|^2_{L^p_tL^\infty_{dr}\,([R,2R])},
\end{eqnarray*}
Since Littlewood--Paley square function inequality fails at $L^\infty_{dr}$ (see Proposition \ref{prop:LP}), thus we use the triangle inequality and get
\begin{equation}
\leq
 \sum_{R\in2^{\Z}}\left(\omega_{\nu}(R)\sum_{N\in 2^\Z}\left\|    \hank\left[e^{it{\rho}} \varphi(\frac{\rho}N)\hank g(\rho) \right]\right\|_{L^p_tL^\infty_{dr}\,([R,2R])}\right)^2.
\end{equation}

Now we recall the following result, see \cite[Proposition 6.2]{CDYZh}.
\begin{proposition}\label{LRE}
Let $p\geq2$, $\varphi\in\mathcal{C}_c^\infty(\R)$ be supported in $I:=[1,2]$, $R>0$ be a dyadic number,  and $\nu\in[-\frac12,\frac12]$. Then, for any $0<\varepsilon\ll 1$ the following estimate holds
\begin{equation}\label{stri-L}
\bigl\|\mathcal{H}_{\nu}\big[e^{ it\rho} h(\rho)\big](r)
\bigr\|_{L^p_tL^\infty_{dr}([R/2,R])}
\lesssim \|h\|_{L^2_{\rho d\rho}(I)}\times
\begin{cases}
R^{\nu-\frac{\varepsilon} 2}&\ \text{if}\  R\lesssim1\\
R^{\frac1p-\frac12}&\ \text{if}\  R\gg1
\end{cases}
\end{equation}
provided $h(\rho)$ is supported in $[1,2]$.
\end{proposition}
Therefore, using a scaling argument and applying $h(\rho)$ with $\varphi(\rho)\hank g(N\rho)$ we have
 \begin{equation}\label{ieq:Hw}
 \begin{aligned}
 &\leq
 \sum_{R\in2^{\Z}}\biggl(\omega_{\nu}(R)\sum_{N\in 2^\Z}N^{2-\frac1p}\left\| \hank\left[e^{it{\rho}} \varphi(\rho)\hank g(N\rho) \right]\right\|_{L^p_tL^\infty_{dr}\,([NR,2NR])}\biggr)^2\\
 &\leq
  \sum_{R\in2^{\Z}}
  \biggl(\sum_{N\in 2^\Z}N^{1-\frac1p} \omega_{\nu}(R)Q(NR)
  \left\| \varphi(\frac\rho{N})\hank g(\rho) \right\|
  _{L^2_{\rho d\rho}}
  \biggr)^2
 \end{aligned}
 \end{equation}
where
\begin{equation}
Q(NR)=
\begin{cases}
  (NR)^{\nu-\frac{\varepsilon} 2}& \ \text{if}\  NR\lesssim 1,\\
  (NR)^{\frac1p-\frac12}& \ \text{if}\ NR\gg1.
\end{cases}
\end{equation}
Due to the fact $p>2$, we have
\begin{equation}\label{q-cond}
\frac1p-\frac12<0.
\end{equation}
We consider two cases of $\nu$ in detail.
\bigskip

$\bullet$ If $\nu= \frac12$, taking $0<\varepsilon<\nu$ we have
\begin{equation}\label{ST}
\mathcal{Q}_1:=\sup_{R} \sum_{N\in2^\Z} Q(NR) <\infty,\quad \mathcal{Q}_2:=\sup_{N} \sum_{R\in2^\Z} Q(NR) <\infty.
\end{equation}
Denote
\begin{equation}
A_{N,\nu}=N^{1 -\frac 1p}
\|\varphi(\frac\rho{N})\hank g(\rho)\|_{L^2_{\rho d\rho}},
\end{equation}
applying the Schur test argument with \eqref{ST} we obtain
\begin{equation*}
  \Bigl\|\sum_{N\in 2^{\Z}}\omega_\nu(R) Q(NR)A_{N,\nu}\Bigr\|_{\ell^{2}_{R}}
  \le
  (\mathcal{Q}_{1}\mathcal{Q}_{2})^{1/2}
  \|A_{N,\nu}\|_{\ell^{2}_{N}}.
\end{equation*}
Thus we have proved
\begin{equation*}
  \|\omega_{\nu}(r)(T_{\nu}g)(t,r)\|_{L^p(\R;L^\infty_{dr})}^{2}
  \le
  C_{\nu}\sum_{N\in2^\Z}|A_{N,\nu}|^2
  =C_{\nu}\|g\|_{\dot H^{1-\frac1p}}^2.
\end{equation*}

\medskip

$\bullet$
If $\nu=-\frac12$ we make the following adjustments to the above argument.
On the one hand, if $NR \lesssim 1$ one has
\begin{equation*}
  \omega_{\nu}(R)Q(NR)=
  \frac{(NR)^{\nu-\frac \epsilon2}}{1+R^{\nu-\epsilon}}=
  \frac{(NR)^{\nu-\epsilon}}
    {N^{\nu-\epsilon}+(NR)^{\nu-\epsilon}}
  (NR)^{\frac \epsilon2}N^{\nu-\epsilon}
  \le (NR)^{\frac \epsilon2}N^{\nu-\epsilon}.
\end{equation*}

On the other hand, for $NR\ge 1$ we have
\begin{equation*}
  \omega_{\nu}(R)Q(NR)=
  \frac{(NR)^{\frac 1p-\frac 12}}{1+R^{\nu-\epsilon}}=
  \frac{N^{\nu-\epsilon}}{N^{\nu-\epsilon}+(NR)^{\nu-\epsilon}}
  (NR)^{\frac 1p-\frac 12}\le
  (NR)^{\frac 1p-\frac 12}N^{\nu-\epsilon}.
\end{equation*}
Then, taking $\epsilon>0$
is sufficiently small so that $1-\frac1p+(\nu-\epsilon)\geq 0$, we obtain the following bounds
\begin{equation*}
  N^{1-\frac1p}\omega_{\nu}(R)Q(NR)\leq
  N^{1-\frac 1p}N^{\nu-\epsilon}\widetilde{Q}(NR)\leq
 \bra{N}^{1-\frac 1p}\widetilde{Q}(NR)
\end{equation*}
where $\bra{N}=(1+N^{2})^{1/2}$ and
\begin{equation}
\tilde{Q}(NR)=
\begin{cases}
  (NR)^{\frac{\epsilon}2}& \ \text{if}\ NR\lesssim 1 \\
  (NR)^{\frac1p-\frac12}&\ \text{if}\  NR\gg1,
\end{cases}
\end{equation}
which implies
\begin{equation*}
  \|\omega_{\nu}(r)T_{\nu}g(t,r)\|_{L^p_tL^\infty_{dr}}^2
  \leq
  \sum_{R\in2^{\mathbb{Z}}}
  \biggl(\sum_{N\in 2^\mathbb{Z}}
  \bra{N}^{1-\frac1p}\tilde{Q}(NR)
    \left\| \varphi(\frac\rho{N})\hank g(\rho) \right\|
    _{L^2_{\rho d\rho}}\biggr)^{2}.
\end{equation*}
We use the previous argument again with replace
$Q(NR)$ by $\tilde{Q}(NR)$ and $A_{N,\nu}$ by $\tilde{A}_{N,\nu}$ which defined by
\begin{equation*}
  \tilde{A}_{N,\nu}=
  \bra{N}^{1-\frac1p}
    \left\| \varphi(\rho/N)\hank f(\rho) \right\|
    _{L^2_{\rho d\rho}}
\end{equation*}
then we get
\begin{equation*}
  \|\omega_{\nu}(r)T_{\nu}f(t,r)\|_{L^p_tL^\infty_{dr}}^2
  \le C_{\nu}
  \sum_{N\in 2^{\mathbb{Z}}}
  |\tilde A_{N,\nu}|^{2}=
  C_{\nu}\|f\|_{ H^{1-\frac1p}}^2.
\end{equation*}
Then we proved the estimate \eqref{stri-D2}. Choosing appropriate $q$ in \eqref{stri-D2} and \eqref{stri-D3} (which will be proved in the Subsection \ref{subs-th3}) respectively we will conclude the proof by using the complex interpolation.

\end{proof}

\section*{Acknowledgements}\label{sec:ackn}

Piero D'Ancona and Zhiqing Yin
are partially supported by the
MUR - PRIN project 2020XB3EFL and
by the INdAM - GNAMPA Project CUP E53C23001670001.
Junyong Zhang is partially supported by  National Natural Science Foundation of China (12171031) and Beijing Natural Science Foundation (1242011).
J.Z. is grateful for the hospitality of the Australian National University when he was visiting Andrew Hassell at ANU.\vspace{0.2cm}

{\bf Conflicts of Interest Statement:}
The authors declare that there are no conflicts of interest relevant to the content of this manuscript. The research was conducted without any commercial or financial relationships that could be construed as a potential conflict of interest.

{\bf Data Availability Statement:}
The data supporting the findings of this study are available from the corresponding author upon reasonable request.

\end{document}